\documentclass{amsart}
\usepackage[utf8]{inputenc}
\usepackage{amssymb}
\usepackage{hyperref}
\usepackage[final]{showkeys} 

\input xy
\xyoption{all}

\theoremstyle{definition}
\newtheorem{mydef}{Definition}[section]
\newtheorem{lem}[mydef]{Lemma}
\newtheorem{lemma}[mydef]{Lemma}
\newtheorem{thm}[mydef]{Theorem}

\newtheorem{cor}[mydef]{Corollary}

\newtheorem{question}[mydef]{Question}
\newtheorem{hypothesis}[mydef]{Hypothesis}

\newtheorem{defin}[mydef]{Definition}
\newtheorem{example}[mydef]{Example}
\newtheorem{exams}[mydef]{Example}
\newtheorem{remark}[mydef]{Remark}
\newtheorem{rem}[mydef]{Remark}
\newtheorem{notation}[mydef]{Notation}
\newtheorem{fact}[mydef]{Fact}

\newcommand{\fct}[2]{{}^{#1}#2}

\newcommand\sqk{\operatorname{Sq}(\ck)}

\newcommand\op{\operatorname{op}}
\newcommand\id{\operatorname{id}}

\newcommand\Set{\operatorname{\bf Set}}

\newcommand\Ban{\operatorname{\bf Ban}}
\newcommand\Hilb{\operatorname{\bf Hilb}}

\newcommand\Gra{\operatorname{\bf Gra}}

\newcommand\colim{\operatorname{colim}}

\newcommand\ca{\mathcal {A}}

\newcommand\cd{\mathcal {D}}

\newcommand\ci{\mathcal {I}}

\newcommand\ck{\mathcal {K}}
\newcommand\cl{\mathcal {L}}

\newcommand\cs{\mathcal {S}}

\newcommand\cx{\mathcal {X}}

\newcommand{\ba}{\bar{a}}
\newcommand{\bb}{\bar{b}}

\newcommand{\bx}{\bar{x}}
\newcommand{\by}{\bar{y}}

\newcommand{\Kmhpp}[2]{{#1}^{#2\text{-lmh}}}
\newcommand{\Kmhp}[1]{\Kmhpp{\K}{#1}}
\newcommand{\Kmh}{\Kmhp{\kappa}}

\newcommand{\Knf}{\K_{\NF}}
\newcommand{\Knfnobf}{K_{\NF}}

\newcommand{\cknf}{\ck_{\NF}}
\newcommand{\leanf}{\leap{\Knf}}

\newcommand{\sea}{\mathfrak{C}}
\newcommand{\ran}[1]{\text{ran}(#1)}

\newcommand{\cf}[1]{\text{cf} (#1)}
\newcommand{\seq}[1]{\langle #1 \rangle}
\newcommand{\rest}{\upharpoonright}

\newcommand{\leap}[1]{\le_{#1}}

\newcommand{\geap}[1]{\ge_{#1}}

\newcommand{\lea}{\leap{\K}}

\newcommand{\gea}{\geap{\K}}

\def\lee{\preceq}

\newcommand{\K}{\mathbf{K}}

\newbox\noforkbox \newdimen\forklinewidth
\forklinewidth=0.3pt \setbox0\hbox{$\textstyle\smile$}
\setbox1\hbox to \wd0{\hfil\vrule width \forklinewidth depth-2pt
 height 10pt \hfil}
\wd1=0 cm \setbox\noforkbox\hbox{\lower 2pt\box1\lower
2pt\box0\relax}
\def\unionstick{\mathop{\copy\noforkbox}\limits}
\newcommand{\nf}{\unionstick}
\newcommand{\nfs}[4]{#2 \nf_{#1}^{#4} #3}

\def\1nf{\unionstick^{(1)}}

\def\2nf{\unionstick^{(2)}}
\def\3nf{\unionstick^{(3)}}

\def\nfm{\overline{\nf}}
\newcommand{\nfcl}[4]{#2 \overset{#4}{\underset{#1}{\overline{\nf}}} #3}

\newcommand{\tp}{\text{tp}}
\newcommand{\gtp}{\text{gtp}}

\newcommand{\gS}{\text{gS}}

\newcommand{\Ll}{\mathbb{L}}

\newcommand{\NF}{\operatorname{NF}}

\newcommand{\tleq}{\trianglelefteq}

\newcommand{\LS}{\text{LS}}

\newcommand{\acl}{\operatorname{acl}}

\title{Forking independence from the categorical point of view}
\date{\today }

\author[Lieberman]{Michael Lieberman}
\email{lieberman@math.muni.cz}
\urladdr{http://www.math.muni.cz/\textasciitilde lieberman/}
\address{Department of Mathematics and Statistics, Faculty of Science, Masaryk University, Brno, Czech Republic}

\author[Rosick\'y]{Ji\v r\'i Rosick\'y}
\email{rosicky@math.muni.cz}
\urladdr{http://www.math.muni.cz/\textasciitilde rosicky/}
\address{Department of Mathematics and Statistics, Faculty of Science, Masaryk University, Brno, Czech Republic}
\thanks{The first and second authors are supported by the Grant Agency of the Czech Republic under the grant P201/12/G028.}

\author[Vasey]{Sebastien Vasey}
\email{sebv@math.harvard.edu}
\urladdr{http://math.harvard.edu/\textasciitilde sebv/}
\address{Department of Mathematics \\ Harvard University \\ Cambridge, Massachusetts, USA}

\begin{document}

\begin{abstract}

  Forking is a central notion of model theory, generalizing linear independence in vector spaces and algebraic independence in fields. We develop the theory of forking in abstract, category-theoretic terms, for reasons both practical (we require a characterization suitable for work in $\mu$-abstract elementary classes, i.e.\ accessible categories with all morphisms monomorphisms) and expository (we hope, with this account, to make forking accessible---and useful---to a broader mathematical audience). In particular, we present an axiomatic definition of what we call a stable independence notion on a category and show that this is in fact a purely category-theoretic axiomatization of the properties of model-theoretic forking in a stable first-order theory.


\end{abstract}

\maketitle

\tableofcontents

\section{Introduction}

\subsection{Background}

Forking is a model-theoretic notion generalizing linear independence in vector spaces and algebraic independence in fields. A central notion of modern model theory, it was developed by Saharon Shelah \cite{shelahfobook} for classes of models axiomatized by a stable first-order theory\footnote{While there are generalizations of the theory of stable forking to certain unstable first-order classes, most notably those axiomatized by NIP and simple theories, the present paper focuses exclusively on the stable case.}. Recall that a first-order theory $T$ has the \emph{order property} if there exists a model $M$ of $T$, a formula $\phi (\bx; \by)$ and a sequence $\seq{\ba_i : i < \omega}$ of finite tuples in $M$ such that for $i, j < \omega$, $M \models \phi[\ba_i; \ba_j]$ if and only if $i < j$. A theory $T$ is \emph{stable} precisely when it does \emph{not} have the order property. For example, both the theory of vector spaces over $\mathbb{Q}$ and the theory of algebraically closed fields of characteristic zero are stable, but the theory of linear orders and that of graphs are not (a formula witnessing the order property inside the random graph, for example, is $\phi (x_1 x_2, y_1 y_2) := x_1 E y_2 \land \neg (x_2 E y_1)$).

Roughly, Shelah defines forking so that a type $p$ over a set $C$ (i.e.\ a set of formulas with parameters from $C$) \emph{does not fork} over a subset $A$ of $C$ if it is a ``generic'' extension of $p \rest A$. In particular, $p$ is essentially determined by $p \rest A$. Thus nonforking\footnote{It is somewhat unfortunate that the negation of forking, ``nonforking'' is the positive notion. Nevertheless, this terminology is now well established.} can be seen as a notion of free extension. There is another way to see nonforking: in his survey on stability theory, Makkai \cite[A.1]{makkai-survey} introduces the \emph{anchor symbol} $\nf$ and defines $\nfs{A}{B}{C}{}$ (when working inside a monster model $\sea$) to mean that for any finite tuples $\bb$ of elements from $B$, $\tp (\bb / A C)$ does not fork over $A$. This notation $\nfs{A}{B}{C}{}$, which can be read as ``$B$ is independent from $C$ over $A$,'' simplifies the statement of some of the properties of forking. For example, \emph{symmetry} can be written as ``$\nfs{A}{B}{C}{}$ if and only if $\nfs{A}{C}{B}{}$''.

Forking is especially well-behaved when the base set $A$ above is a model of the theory\footnote{To study forking over sets and not just models, one can consider a certain category of sufficiently algebraically closed sets (Example \ref{indep-examples}(\ref{fo-example})) in which the nonforking amalgams determine the behavior of forking over any set.}. In fact, to understand it category-theoretically, it is useful to consider the case in which \emph{all} the sets under consideration are models, specifically $\nfs{M_0}{M_1}{M_2}{M_3}$, with $M_0 \preceq M_\ell \preceq M_3$, $\ell = 1,2$. Here, we write $M_3$ for an ambient model inside which all the types are computed. One can also view such a quadruple as a commutative diagram of embeddings, also known as an {\it amalgam}.  In this sense, we may identify a nonforking notion over models with a particular choice of such amalgams. Note that consideration of nonforking amalgams plays an important role in model theory, as it leads to definitions such as that of an \emph{independent system of models} (see \cite[A.11]{makkai-survey} or \cite{sh87a, sh87b}). This is for example a key concept in both the statement and proof of Shelah's main gap theorem \cite{main-gap-ams}, a celebrated achievement of first-order classification theory.

\subsection{Main questions and earlier work}

The present paper seeks to answer the following questions:

\begin{enumerate}
\item What are the basic category-theoretic properties of nonforking amalgams in the category of models of a stable first-order theory? More precisely, we would like a list of properties that are:

  \begin{enumerate}
  \item Invariant under equivalence of categories, e.g.\ they should not depend on what underlying concrete functor we use to represent the category.
  \item Canonical: in any reasonable category, there should be at most one notion satisfying the properties of nonforking amalgamation.
  \end{enumerate}
\item In what other categories is there a notion satisfying these properties?   
\end{enumerate}

There has already been a substantial amount of work on more purely model-theoretic versions of these questions. Consider, in particular, Harnik and Harrington \cite{hh84}, which characterizes forking in a stable first-order theory by a list of four axioms on a relation of inclusion between types. These axioms are, however, not suitably category-theoretic, as they depend on seeing types as sets of formulas. Recently, Boney, Grossberg, Kolesnikov, and the third author \cite{bgkv-apal} characterized stable forking in the general framework of abstract elementary classes (AECs), encompassing first-order theories but also classes axiomatized by infinitary logics such as $\Ll_{\infty, \omega}$. The characterization of \cite{bgkv-apal} is phrased in terms of an anchor relation $\nf$, but still uses the underlying set representations of the objects in the category.

The second question was considered early on by Shelah, with partial results in homogeneous model theory \cite{sh3}, $\Ll_{\omega_1, \omega}$ \cite{sh87a, sh87b}, universal classes \cite{sh300-orig}, and AECs (e.g.\ in his two-volume book \cite{shelahaecbook, shelahaecbook2}). In fact, there is a growing body of literature on forking-like notions in AECs. We refer the reader to the survey of Boney and the third author \cite{bv-survey-bfo}, but let us mention in particular the work of Boney and Grossberg \cite{bg-apal}, which generalizes work of Makkai and Shelah \cite{makkaishelah} from $\Ll_{\kappa, \omega}$ ($\kappa$ a strongly compact cardinal) to AECs, and builds a global forking-like independence notion on a subclass of sufficiently saturated models of the AEC. Interestingly, the subclass is not itself an AEC, but is still closed under $\kappa$-directed unions (as opposed to arbitrary directed unions). This was one motivation for developing $\kappa$-AECs \cite{mu-aec-jpaa}. This is a very broad framework for model theory. In fact, it has a category-theoretic equivalent: per Fact \ref{mu-aec-acc}, a $\kappa$-AEC is exactly (up to equivalence of category) an accessible category whose morphisms are monomorphisms (see Makkai-Paré \cite{makkai-pare} or Adámek-Rosický \cite{adamek-rosicky} on accessible categories and their broader relationship to model theory).

\subsection{Main results}

It therefore seems natural to investigate forking in arbitrary accessible categories (perhaps with all morphisms monomorphisms). The present paper makes the following contributions:

\begin{itemize}
\item We define when a category has what we call a \emph{stable independence notion} (Definition \ref{stable-def}). Roughly, this is a class of distinguished squares that satisfies, in particular, an existence property, a uniqueness property (a weakening of the definition of a pushout), as well as a transitivity property (corresponding to being closed under composition in a double categorical sense). The transitivity property makes the class of independent squares into a category, and we require that this category be \emph{accessible}.
\item We show that this is the desired purely category-theoretic axiomatization of forking: in a $\mu$-AEC $\K$, being stable (and, specifically, the accessibility of the category of independent squares) corresponds to having certain local character properties well known to model theorists (Theorem \ref{accessible-charact}). Moreover this axiomatization is canonical, assuming that the class has chain bounds (that is, any increasing chain of models has an upper bound, see Definition \ref{directed-def}).  This result, Theorem \ref{canon-thm}, generalizes \cite{bgkv-apal} to $\mu$-AECs.
\item Working purely abstractly, we exhibit a connection between forking and {\it effective unions} (an exactness property introduced by Barr \cite{effective-unions}). Specifically, we show that, if we start with a locally presentable and coregular category $\ck$ that has effective unions then $\ck_{reg}$, the subcategory of $\ck$ containing just the regular monomorphisms of $\ck$, has a stable independence notion, see Theorem \ref{indep}. In particular, this covers both Grothendieck toposes and Grothendieck abelian categories. That forking occurs in these contexts seems not to have been recognized before (although it has long been known that forking occurs in classes axiomatized by first-order theories of modules, see \cite{prest}).
\item Assuming a large cardinal axiom, we characterize precisely when a stable independence notion exists in any $\mu$-AEC with chain  bounds (and hence in any accessible category with chain bounds whose morphisms are monomorphisms). This is Corollary \ref{indep-build-cor}: such a $\mu$-AEC has a cofinal subclass with a stable independence notion if and only if it does not have a certain order property. This implies that the usual ``syntactic'' definition of stability (note that the usual definition in terms of counting types is too weak in this context, see Example \ref{stability-op-example}) is equivalent to a purely category-theoretic statement. Thus model-theoretic stability is invariant under equivalence of category\footnote{For the category of models of a first-order theory, this can also be seen using Shelah's saturation spectrum theorem \cite[VIII.4.7]{shelahfobook}, see also \cite{rosicky-sat-jsl}. However no such saturation spectrum theorem is known in arbitrary $\mu$-AECs.}. As a philosophical remark, Shelah \cite[p.~23]{shelahaecbook} argues that classification-theoretic dividing lines should have both an ``internal'' and an ``external'' characterization. If we interpret ``external'' as ``invariant under equivalence of category'' and ``internal'' as ``a property satisfied by a fixed model in the class'', then we see here this principle in action and obtain evidence that there is a ``stability-like'' dividing line in the general framework of accessible categories.
\end{itemize}

\subsection{Notes}

In a sense, this paper falls naturally into two parts: after a brief review of some of the basic concepts that will be used (Section \ref{prelim-sec}), we devote Sections \ref{stable-indep-categ} through \ref{coreg-sec} to the development of an analogue of stable (or nonforking) independence suited to a general category and, working purely abstractly, give conditions on a category under which such a relation is (a) guaranteed to exist, and (b) to take a particular recognizable form, e.g. the independent squares are precisely the pullback squares. This first half is intended to be congenial to a broad mathematical audience, and may be of particular use to those who have previously been reluctant to wade into the details of model-theoretic nonforking.

In the second half, we shift our attention to $\mu$-AECs, which, being concrete, come with a great deal more machinery, and thus allow the development of a richer theory---the results and proofs here are more recognizably model-theoretic, and significantly more technical.  We note, however, that $\mu$-AECs are accessible categories with all morphisms monomorphisms, and vice versa---so, in fact, we are working in a vastly more general context than the coregular locally presentable categories of, say, Section \ref{effect-sec}. Section \ref{categ-mu-aec} gives some tools to move from an arbitrary category to a $\mu$-AEC by restricting its class of morphisms: this provides a bridge between the paper's two halves. In Section \ref{mt-mu-aec}, we give some basic model-theoretic tools for use in $\mu$-AECs. In section \ref{stable-mu-aec}, we consider what the properties of stable independence look like in a $\mu$-AEC and show that they are equivalent to the more model-theoretic local character properties of forking in a stable first-order theory. In Section \ref{canon-sec}, we prove that stable independence notions are canonical, symmetric, and imply failure of a certain order property. In Section \ref{noop-sec}, we reverse this and show (assuming a large cardinal axiom) that failure of an order property implies there is a stable independence notion on a subclass of saturated models.

Throughout this paper, we assume the reader is familiar with basic category theory as presented e.g.\ in \cite{joy-of-cats}. More particularly, we will spend much of our time in accessible, locally presentable, or locally multipresentable categories (see \cite{adamek-rosicky} for further details).  For connections between locally multipresentable categories---and locally polypresentable categories, which also make a brief appearance---and abstract model theory, we point readers to \cite{multipres-pams}.  We also assume some familiarity with $\mu$-AECs and their relationship with accessible categories \cite{mu-aec-jpaa}.

We use the following notational conventions: we write $K$ for a class of $\tau$-structures and $\K$ (boldface) for a pair $(K, \lea)$, where $\lea$ is a partial order. We write $\ck$ (script) for a category. We will abuse notation and write $M \in \K$ instead of $M \in K$. For a structure $M$, we write $U M$ for its universe, and $|U M|$ for the cardinality of its universe. We write $M \subseteq N$ to mean that $M$ is a substructure of $N$. For $\alpha$ an ordinal, we let $\fct{<\alpha}{A}$ (respectively, $\fct{\alpha}{A}$) denote the set of sequences of length less than (respectively, equal to) $\alpha$ with elements from the set $A$. We will abuse this notation slightly, writing $\fct{<\alpha}{M}$ in place of $\fct{<\alpha}{U M}$.

\subsection{Acknowledgments}

We thank Marcos Mazari Armida for very detailed feedback which greatly improved the paper. We also would like to thank the referee for helpful comments.

\section{Preliminaries}\label{prelim-sec}

Intuitively, an accessible category is a category with all  sufficiently directed colimits and such that every object can be written as a highly directed colimit of ``small'' objects.  Here ``small'' is interpreted in terms of {\em presentability}, a notion of size that makes sense in an arbitrary (potentially non-concrete) category. It is important to note that this notion, which may well appear unfamiliar, restricts to precisely what it should in familiar cases: in the category of sets, a set is $\lambda$-presentable if and only if its cardinality is less than $\lambda$; in an AEC $\K$, the same is true for all $\lambda>\LS(\K)$. 

\begin{defin}\label{acc-def}
  Let $\ck$ be a category and let $\lambda$ be a regular cardinal.

  \begin{enumerate}
  \item An object $M$ is \emph{$\lambda$-presentable} if its hom-functor $\ck(M,-):\ck\to\Set$ preserves $\lambda$-directed colimits. Put another way, $M$ is $\lambda$-presentable if for any morphism $f:M\to N$ with $N$ a $\lambda$-directed colimit $\langle \phi_\alpha:N_\alpha\to N\rangle$, $f$ factors essentially uniquely through one of the $N_\alpha$, i.e.\ $f=\phi_\alpha f_\alpha$ for some $f_\alpha:M\to N_\alpha$.
  \item $\ck$ is \emph{$\lambda$-accessible} if it has $\lambda$-directed colimits and $\ck$ contains a set $S$ of $\lambda$-presentable objects such that every object of $\ck$ is a $\lambda$-directed colimit of objects in $S$.
  \item $\ck$ is \emph{accessible} if it is $\lambda'$-accessible for some regular cardinal $\lambda'$.
  \end{enumerate}
\end{defin}

We will often quote results on accessible categories from \cite{adamek-rosicky}.

Recall from \cite[\S2]{mu-aec-jpaa} that a \emph{($\mu$-ary) abstract class} is a pair $\K = (K, \le)$ such that $K$ is a class of structures in a fixed $\mu$-ary vocabulary $\tau = \tau (\K)$, and $\le$ is a partial order on $K$ that refines the $\tau$-substructure relation, with $K$ and $\le$ closed under $\tau$-isomorphism. We say that such a $\K$ is \emph{coherent} if $M_0 \subseteq M_1 \lea M_2$ and $M_0 \lea M_2$ implies $M_0 \lea M_1$.

In any abstract class $\K$, there is a natural notion of morphism: we say that $f: M \rightarrow N$ is a \emph{$\K$-embedding} if $f$ is an isomorphism from $M$ onto $f[M]$ and $f[M] \lea N$. We can see an abstract class and its $\K$-embeddings as a category. In fact (see \cite[\S2]{mu-aec-jpaa}), an abstract class is a replete and iso-full subcategory of the category of $\tau$-structures with injective homomorphisms (insisting, as is customary in model theory, that relation symbols be reflected as well as preserved). Thus we also think of $\K$ as a (concrete) category. Saying that $\K$ has \emph{concrete $\mu$-directed colimits} amounts to saying that for any $\mu$-directed system in $\K$, the union of the system is its colimit. That is, $\K$ satisfies the chain axioms of $\mu$-AECs. In fact, let us now recall the definition of a $\mu$-AEC from \cite[2.2]{mu-aec-jpaa}:

\begin{defin}
  Let $\mu$ be a regular cardinal. An abstract class $\K$ is a \emph{$\mu$-abstract elementary class} (or \emph{$\mu$-AEC} for short) if it satisfies the following three axioms:

  \begin{enumerate}
  \item Coherence: for any $M_0, M_1, M_2 \in \K$, if $M_0 \subseteq M_1 \lea M_2$ and $M_0 \lea M_2$, then $M_0 \lea M_1$.
  \item Chain axioms: if $\seq{M_i : i \in I}$ is a $\mu$-directed system in $\K$, then:
    \begin{enumerate}
    \item $M := \bigcup_{i \in I} M_i$ is in $\K$.
    \item $M_i \lea M$ for all $i \in I$.
    \item If $M_i \lea N$ for all $i \in I$, then $M \lea N$.
    \end{enumerate}
  \item Löwenheim-Skolem-Tarski (LST) axiom: there exists a cardinal $\lambda = \lambda^{<\mu} \ge |\tau (\K)| + \mu$ such that for any $M \in \K$ and any $A \subseteq U M$, there exists $M_0 \in \K$ with $M_0 \lea M$, $A \subseteq U M_0$, and $|U M_0| \le |A|^{<\mu} + \lambda$. We write $\LS (\K)$ for the least such $\lambda$.
  \end{enumerate}
\end{defin}

Note that when $\mu = \aleph_0$, we recover Shelah's definition of an AEC from \cite{sh88}. The connection between $\mu$-AECs and accessible categories is given by \cite[\S4]{mu-aec-jpaa}:

\begin{fact}\label{mu-aec-acc}
  If $\K$ is a $\mu$-AEC, then it is an $\LS (\K)^+$-accessible category with all $\mu$-directed colimits whose morphisms are monomorphisms. Conversely, any $\mu$-accessible category whose morphisms are monomorphisms is equivalent to a $\mu$-AEC.
\end{fact}

We end this section by briefly recalling that we can define a notion of Galois (orbital) type in any abstract class. This notion of type was introduced by Shelah for AECs but we use the notation and definitions from \cite[\S2]{sv-infinitary-stability-afml}. Loosely speaking, it is defined as the finest notion of type preserving $\K$-embeddings.  More precisely, the Galois type of a sequence $\ba$ over a structure $M\in\K$ computed in an extension $N\gea M$, denoted $\gtp (\ba / M; N)$, is the equivalence class of the triple $(M,\ba,N)$ under the (transitive closure of) the relation $\sim$, where $(M,\ba,N)\sim (M,\ba',N')$ if there is an amalgam 

$$  
      \xymatrix@=3pc{
        N \ar@{.>}[r]^{g} & \bar{N} \\
        M \ar [u] \ar [r] &
        N' \ar@{.>}[u]_{g'}
      }
      $$ 
with $g(\ba)=g'(\ba')$ (see, e.g., \cite[2.16]{sv-infinitary-stability-afml}).  We denote by $\gS^{<\infty} (M)$ the set of Galois types of any length over $M$; that is,

$$
\gS^{<\infty} (M) := \{\gtp (\ba / M; N) \mid M \lea N, \ba \in \fct{<\infty}{N}\}.
$$

We will also use variations such as $\gS^{\alpha} (M)$ (the length is restricted to be $\alpha$) or $\gS^{<\infty} (B; N)$ (the base set is $B$ and we only look at types of elements inside $N$).

\section{Stable independence in an arbitrary category}\label{stable-indep-categ}

Throughout this section, we assume:

\begin{hypothesis}
  We work inside a fixed category $\ck$.
\end{hypothesis}

The goal of this section is to axiomatize stable amalgams as a particular category of commutative squares in $\ck$ (see Definition \ref{k-nf-def}). In fact, although we will not stress this perspective, we will identify this notion with a double subcategory of $\sqk$, the usual double category of commutative squares in $\ck$.

\begin{defin}\label{sim-def}
  Let $(f_1 : M_0 \rightarrow M_1, f_2 : M_0 \rightarrow M_2)$ be a span in $\ck$.

  \begin{enumerate}
    \item An \emph{amalgam} of $(f_1, f_2)$ is a cospan $(g_1 : M_1 \rightarrow N, g_2 : M_2 \rightarrow N)$ such that the following diagram commutes:

      $$  
      \xymatrix@=3pc{
        M_1 \ar@{.>}[r]^{g_1} & N \\
        M_0 \ar [u]^{f_1} \ar [r]_{f_2} &
        M_2 \ar@{.>}[u]_{g_2}
      }
      $$

      An \emph{amalgamation diagram} is a quadruple $(f_1, f_2, g_1, g_2)$ such that $(f_1, f_2)$ is a cospan and $(g_1, g_2)$ is an amalgam thereof.
    \item We say that $\ck$ has the \emph{amalgamation property} if every span has an amalgam.
    \item Two amalgams $g_1^a: M_1 \rightarrow N^a, g_2^a : M_2 \rightarrow N^a$, $g_1^b : M_1 \rightarrow N^b, g_2^b : M_2 \rightarrow N^b$ of $(f_1, f_2)$ are \emph{equivalent} (written $(f_1, f_2, g_1^a, g_2^a) \sim^\ast (f_1, f_2, g_1^b, g_2^b)$ if there exists $N$ and $g^a$, $g^b$ making the following diagram commute:

        \[
        \xymatrix{ & N^b \ar@{.>}[r]^{g^b} & N_\ast \\
    M_1 \ar[ru]^{g_1^b} \ar[rr]|>>>>>{g_1^a} & & N^a \ar@{.>}[u]_{g^a} \\
    M_0 \ar[u]^{f_1} \ar[r]_{f_2} & M_2 \ar[uu]|>>>>>{g_2^b}  \ar[ur]_{g_2^a} & \\
  }
        \]
      \item Let $\sim$ be the transitive closure of $\sim^\ast$.
  \end{enumerate}
\end{defin}

\begin{remark}
  If $\ck$ has the amalgamation property, then $\sim^\ast$ is already transitive \cite[4.3]{jrsh875}.
\end{remark}

The idea of $\sim$ is to identify amalgams whose underlying spans look the same. From this perspective, an independence relation simply consists in the specification of a particular way to amalgamate each span.

\begin{defin}
  An \emph{independence relation (on $\ck$)} is a set $\nf$ of amalgamation diagrams that is closed under $\sim$. We write $\nf (f_1, f_2, g_1, g_2)$ instead of $(f_1, f_2, g_1, g_2) \in \nf$. If an amalgamation diagram $(f_1, f_2, g_1, g_2)$ is in $\nf$, we call it an \emph{$\nf$-independent diagram} (or just an \emph{independent diagram} when $\nf$ is clear from context).
\end{defin}
\begin{remark}
  We will use the terms \emph{independence relation} and \emph{independence notion} completely interchangeably.
\end{remark}

\begin{notation}\label{stablediag}
	We note a certain utility, too, in a diagrammatic representation analogous to that for pullbacks and pushouts; that is, one may wish to represent the assertion $\nf (f_1, f_2, g_1, g_2)$ by the annotated diagram
	$$  
      \xymatrix@=3pc{
        M_1 \ar@{}[dr]|{\nf}\ar[r]^{g_1} & N \\
        M_0 \ar [u]^{f_1} \ar [r]_{f_2} &
        M_2 \ar[u]_{g_2}
      }
      $$
    This points to the fact that, in axiomatizing abstract independence, we will be delineating the properties of a (double) category of such squares---this is made precise in Definition~\ref{k-nf-def} below.
\end{notation}

When we are working in a concrete class and the morphisms are simply inclusions of strong substructures, we will employ the obvious notational shortcut:

\begin{notation}
  For an independence relation $\nf$ on an abstract class $\K$, we write $\nfs{M_0}{M_1}{M_2}{M_3}$ (or $\nf (M_0, M_1, M_2, M_3)$) if $M_0 \lea M_\ell \lea M_3$ for $\ell = 1,2$ and $\nf (i_{0, 1}, i_{0, 2}, i_{1, 3}, i_{2, 3})$, where $i_{l, k}$ is the $\K$-embedding from 
  $M_l$ to $M_k$.
\end{notation}

\begin{defin}\label{invariant-def}
  An independence relation $\nf$ is \emph{invariant} if it is invariant under isomorphisms of amalgamation diagrams (in the expected sense). 
\end{defin}

One can always switch the left and right hand side of $\nf$ and obtain another independence relation. Thus it is natural to define: 

\begin{defin}
  For $\nf$ an independence relation on $\ck$, we define the \emph{dual} of $\nf$, denoted $\nf^d$ by $\nf^d (f_1, f_2, g_1, g_2)$ if and only if $\nf (f_2, f_1, g_2, g_1)$. We say that $\nf$ is \emph{symmetric} if $\nf = \nf^d$.
\end{defin}

The following property is a strengthening of the amalgamation property: it asks that any span have an \emph{independent} amalgam.

\begin{defin}
  We say that $\nf$ has the \emph{existence property} (or just \emph{has existence}) if for any cospan $(f_1, f_2)$, there is an amalgam $(g_1, g_2)$ such that $\nf (f_1, f_2, g_1, g_2)$ (so in particular, $\ck$ has the amalgamation property).
\end{defin}
\begin{remark}\label{existence-dual}
  If $\nf$ has the existence property, then $\nf^d$ has the existence property.
\end{remark}

The existence property implies that the base is independent over itself. More generally:

\begin{lem}\label{fd-rmk}
  Assume that $\nf$ has the existence property. Consider the commutative diagram

  $$  
  \xymatrix@=3pc{
    M_1 \ar[r]^{g_1} & M_3 \\
    M_0 \ar [u]^{f_1} \ar [r]_{f_2} &
    M_2 \ar[u]_{g_2}
  }
  $$

  If either $f_1$ or $f_2$ is an isomorphism, then the diagram is independent.
\end{lem}
\begin{proof}
  In light of \ref{existence-dual}, it suffices to prove the result in case $f_1$ is an isomorphism.
  
  Assume that $f_1$ is an isomorphism and apply existence to the cospan $(f_1, f_2)$. We obtain an amalgam $(h_1: M_0 \rightarrow M_3', h_2: M_2 \rightarrow M_3')$ such that $\nf (f_1, f_2, h_1, h_2)$. Now because $f_1$ is an isomorphism, $(g_1, g_2)$ is equivalent to $(h_1, h_{2})$. Indeed, any amalgam of $(g_2, h_2)$ will witness the equivalence. Since $\nf$ is (by definition of an independence relation) closed under $\sim$, we also have that $\nf (f_1, f_2, g_1, g_2)$.
\end{proof}

The statement of the uniqueness property, below, may seem unusual. The idea is that we want every span to have an independent amalgam which is unique, not up to isomorphism (as we always want e.g.\ to be able to grow the ambient model) but up to equivalence of amalgams:

\begin{defin}
  We say that $\nf$ has the \emph{uniqueness property} if whenever $\nf (f_0, f_1, g_1, g_2)$ and $\nf (f_0, f_1, g_1', g_2')$, we have that $(f_0, f_1, g_1, g_2) \sim (f_0, f_1, g_1', g_2')$.
\end{defin}
\begin{remark}
  If $\nf$ has the uniqueness property, then $\nf^d$ has the uniqueness property.
\end{remark}

In order to coherently compose independent squares, the following property is key:

\begin{defin}
  Let $\nf$ be an independence relation. We say that $\nf$ is \emph{right transitive} if whenever we have:

  $$  
  \xymatrix@=3pc{
    M_1 \ar[r]^{f_{1,3}} & M_3 \ar[r]^{f_{3,5}}  & M_5 \\
    M_0 \ar [u]^{f_{0,1}} \ar [r]_{f_{0,2}} & M_2 \ar[u]_{f_{2,3}} \ar[r]_{f_{2, 4}} & M_4 \ar[u]_{f_{4,5}}
  }
  $$

  with $\nf(f_{0, 1}, f_{0, 2}, f_{1, 3}, f_{2, 3})$ and $\nf(f_{2,3}, f_{2,4 }, f_{3, 5}, f_{4, 5})$, then it is also the case that $\nf(f_{0,1},  f_{2, 4}\circ f_{0, 2}, f_{3, 5} \circ f_{1, 3}, f_{4, 5})$.

  We say that $\nf$ is \emph{left transitive} if $\nf^d$ is right-transitive. We say that $\nf$ is \emph{transitive} if it is both left and right transitive.
\end{defin}

Diagrammatically, right transitivity means precisely that the situation in the leftmost diagram below implies that on the right:
$$  
  \xymatrix@=3pc{
    M_1 \ar[r]^{f_{1,3}}\ar@{}[dr]|{\nf} & M_3\ar@{}[dr]|{\nf} \ar[r]^{f_{3,5}}  & M_5 \\
    M_0 \ar [u]^{f_{0,1}} \ar [r]_{f_{0,2}} & M_2 \ar[u]_{f_{2,3}} \ar[r]_{f_{2, 4}} & M_4 \ar[u]_{f_{4,5}}
  } \hspace{15 mm}
  \xymatrix@=3pc{
    M_1 \ar[r]^{f_{3,5}f_{1,3}}\ar@{}[dr]|{\nf} & M_5 \\
    M_0 \ar [u]^{f_{0,1}} \ar [r]_{f_{2,4}f_{0,2}} & M_4 \ar[u]_{f_{2,3}}
  }
  $$
In other words, right transitivity guarantees that the collection of independent squares is closed under horizontal composition. Dually, left transitivity gives closure under vertical composition. While this (in conjunction with the preceding properties) gives us the structure of double category, we opt instead, for practical reasons, to describe the category of independent squares in more straightforward 1-categorical terms:

\begin{defin}\label{k-nf-def}
  Given a \emph{right transitive} independence relation $\nf$ with existence, define $\cknf$ (where $\NF$ stands for ``nonforking'') to be the following category:

  \begin{enumerate}
  \item Its objects are morphisms $f: M_1 \rightarrow M_2$ in $\ck$.
  \item A morphism in $\ck_{\NF}$ from $f: M_1 \rightarrow M_2$ to $g: N_1 \rightarrow N_2$ is a pair $(h_1 : M_1 \rightarrow N_1, h_2 : M_2 \rightarrow N_2)$ such that the following is an independent diagram:
    
      $$  
      \xymatrix@=3pc{
        M_2 \ar[r]^{h_2} & N_2 \\
        M_1 \ar [u]^{f} \ar [r]_{h_1} &
        N_1 \ar[u]_{g}
      }
      $$
      
    \item Composition of morphisms is defined as expected. Note that right transitivity exactly gives that $\cknf$ is closed under composition and Lemma \ref{fd-rmk} gives the existence of an identity morphism.
  \end{enumerate}
\end{defin}

Clearly, $\cknf$ is a subcategory of the category $\ck^2$ of morphisms in $\ck$ with the same objects as $\ck^2$.

\begin{remark}\label{full-subcat-rmk}
  Let $\nf$ be a right transitive independence relation with the existence property. Then (using Lemma \ref{fd-rmk}) $\ck$ is isomorphic to a full subcategory of $\cknf$.
\end{remark}

We can use $\cknf$ to give a quick proof that invariance (Definition \ref{invariant-def}) follows from the properties defined so far.

\begin{lem}
  Let $\nf$ be an independence relation. If $\nf$ is right transitive and has existence\footnote{It suffices to assume the conclusion of Lemma \ref{fd-rmk}.}, then $\nf$ is invariant. 
\end{lem}
\begin{proof}
  Assume we are given two isomorphic amalgamation diagrams $(f_1, f_2, g_1, g_2)$, $(f_1', f_2', g_1', g_2')$, where the first one is independent. To see that the second diagram is independent, it suffices to see that the morphism $H = (f_2', g_1')$ from $f_1'$ to $g_2'$  (a morphism of $\ck^2$) is in fact a morphism in $\cknf$. Now $f_1'$ is (in $\ck^2$) isomorphic to $f_1$, and $g_2'$ is isomorphic to $g_2$. Moreover, these isomorphisms are in $\cknf$ by Lemma \ref{fd-rmk}. Composing these isomorphisms with $(f_2, g_1)$ in $\cknf$, we obtain $H$, which therefore must also be in $\cknf$.
\end{proof}

Before studying $\cknf$ further, we note that a nice-enough independence relation will be monotonic in the following sense:

\begin{defin}\label{monot-def}
  Let $\nf$ be an independence relation. We say that $\nf$ is \emph{right monotonic} if whenever we have:

    $$  
  \xymatrix@=3pc{
    M_1 \ar[rr]^{f_{1,4}} &  & M_4 \\
    M_0 \ar [u]^{f_{0,1}} \ar [r]_{f_{0,2}} & M_2 \ar[r]_{f_{2, 3}} & M_3 \ar[u]_{f_{3,4}}
  }
  $$

if $(f_{0,1}, f_{2, 3} \circ f_{0, 2}, f_{1, 4}, f_{3, 4})$ is an independent diagram, then $(f_{0, 1},f_{0, 2},f_{1, 4},f_{3, 4} \circ f_{2, 3})$ is an independent diagram.

  We say that $\nf$ is \emph{left monotonic} if $\nf^d$ is right monotonic. We say that $\nf$ is \emph{monotonic} if it is both left and right monotonic.
\end{defin}

\begin{lem}\label{monot-lem}
  Let $\nf$ be an independence relation with existence and uniqueness. If $\nf$ is right transitive, then $\nf$ is right monotonic.
\end{lem}
\begin{proof}
  We start with the diagram from Definition \ref{monot-def} and assume that 
  $$(f_{0,1}, f_{2, 3} \circ f_{0, 2}, f_{1, 4}, f_{3, 4})$$ is an independent diagram. Using existence, pick $g_{1,4} : M_1 \rightarrow M_4'$, $g_{2,4}: M_2 \rightarrow M_4'$ such that $\nf (f_{0,1}, f_{0, 2}, g_{1, 4}, g_{2, 4})$. Now pick $g_{4,5} : M_4' \rightarrow M_5'$, $g_{3,5}: M_3 \rightarrow M_5'$ such that $\nf(g_{2, 4}, f_{2, 3}, g_{4,5}, g_{3, 4})$. By right transitivity, $\nf (f_{0,1}, f_{2,3} \circ f_{0,2}, g_{4, 5} \circ g_{1, 4}, g_{3,5})$. By uniqueness, we obtain the following picture:

  $$
  
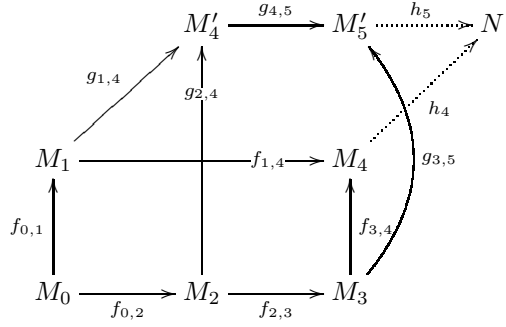
\xymatrix@=3pc{
    & M_4' \ar[r]^{g_{4,5}} & M_5' \ar @{.>} [r]^{h_5} & N \\
    M_1 \ar[ur]^{g_{1,4}} \ar[rr]|>>>>>>>>{f_{1,4}} &  & M_4 \ar @{.>} [ur]_{h_4} \\
    M_0 \ar [u]^{f_{0,1}} \ar [r]_{f_{0,2}} & M_2 \ar[uu]|>>>>>>>{g_{2,4}} \ar[r]_{f_{2, 3}} & M_3 \ar[u]_{f_{3,4}} \ar @/_2pc/ [uu]_{g_{3,5}}
  }
  $$

  This shows that the amalgams $(f_{0,1}, f_{0,2}, g_{1, 4}, g_{2, 4})$ and $(f_{0, 1}, f_{0, 2}, f_{1, 4}, f_{3,4} \circ f_{2,3})$ are equivalent. Since the first is an independent diagram, the second also is. 
\end{proof}

The following variation will also be useful:

\begin{lem}\label{descent}
Let $\nf$ be a right monotonic independence relation and consider
$$  
  
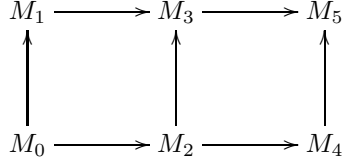
\xymatrix@=3pc{
    M_1 \ar[r]^{} & M_3 \ar[r]^{}  & M_5 \\
    M_0 \ar [u]^{} \ar [r]_{} & M_2 \ar[u]_{} \ar[r]_{} & M_4 \ar[u]_{}
  }
  $$
where the outer rectangle is independent. Then the left square is independent.
\end{lem}
\begin{proof}
Since $\nf$ is right monotonic, the square
$$  
      
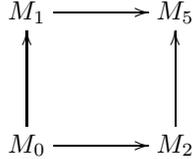
\xymatrix@=3pc{
      M_1   \ar[r]^{} & M_5 \\
        M_0 \ar [u]^{} \ar [r]_{} &
        M_2 \ar[u]_{}
      }
      $$
is independent. Since this square is equivalent to the left square, the latter is independent. The equivalence is documented by     
the diagram

        
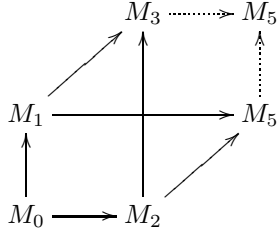
\[
        \xymatrix{ & M_3 \ar@{.>}[r]^{} & M_5 \\
    M_1 \ar[ru]^{} \ar[rr]|>>>>>{} & & M_5 \ar@{.>}[u]_{} \\
    M_0 \ar[u]^{} \ar[r]_{} & M_2 \ar[uu]|>>>>>{}  \ar[ur]_{} & \\
  }
        \]
\end{proof}

We can also define a monotonicity property with respect to the base:

\begin{defin}\label{base-monot-def}
  Let $\nf$ be an independence relation. We say that $\nf$ is \emph{right base-monotonic} if whenever we have:

  $$  
  \xymatrix@=3pc{
    M_1 \ar[rr]^{f_{1,4}} &  & M_4 \\
    M_0 \ar [u]^{f_{0,1}} \ar [r]_{f_{0,2}} & M_2 \ar[r]_{f_{2, 3}} & M_3 \ar[u]_{f_{3,4}}
  }
  $$

  and the outer square $(f_{0,1}, f_{2,3} \circ f_{0,2}, f_{1,4}, f_{3,4})$ is independent, then there exists $f_{4,4'}: M_4 \rightarrow M_4'$ and $f_{1,1'}: M_1 \rightarrow M_1'$, $f_{2, 1'}: M_2 \rightarrow M_1'$, $f_{1', 4'} : M_1' \rightarrow M_4'$ such that $(f_{2,1'}, f_{2,3}, f_{1', 4'}, f_{4, 4'} \circ f_{3,4})$ is independent and the diagram below commutes:

  \[
  \xymatrix@=3pc{
    & M_1' \ar@{.>}[r]_{f_{1', 4'}} & M_4' \\
    M_1 \ar@{.>}[ur]^{f_{1, 1'}} \ar[rr]|>>>>>>>>>{f_{1,4}} &  & M_4 \ar@{.>}[u]_{f_{4, 4'}} \\
    M_0 \ar [u]^{f_{0,1}} \ar [r]_{f_{0,2}} & M_2 \ar[r]_{f_{2, 3}} \ar@{.>}[uu]|>>>>>>{f_{2, 1'}} & M_3 \ar[u]_{f_{3,4}}
  }
  \]

  As usual, $\nf$ left base-monotonic means that $\nf^d$ is right base-monotonic, and base-monotonic means both left and right.
\end{defin}

Any nice-enough independence relation is base-monotonic. More precisely:

\begin{lem}\label{base-monot-lem}
  If $\nf$ is right transitive and has uniqueness and existence, then $\nf$ is right base-monotonic.
\end{lem}
\begin{proof}
  This is similar to (E)(b) in \cite[III.9.6]{shelahaecbook}. In detail, start with the setup of Definition \ref{base-monot-def}. By Lemma \ref{monot-lem}, the square $(f_{0,1}, f_{0,2}, f_{1,4}, f_{3,4} \circ f_{2,3})$ is independent. By existence, find $f_{3,3^\ast}: M_3 \rightarrow M_3^\ast$ and $f_{4,3^\ast} : M_4 \rightarrow M_3^\ast$ such that the square $(f_{3,4} \circ f_{2,3}, f_{2,3}, f_{4, 3'}, f_{3, 3^\ast})$ is independent. By right transitivity, the square $(f_{0, 1}, f_{2,3} \circ f_{0,2}, f_{4, 3^\ast} \circ f_{1,4}, f_{3, 3^\ast})$ is independent. By uniqueness, this square is equivalent to $(f_{0,1}, f_{2, 3} \circ f_{0,2}, f_{1,4}, f_{3,4})$. Amalgamate these two squares: let $f_{4, 4^\ast}: M_4 \rightarrow M_4^\ast$ and $f_{3^\ast, 4^\ast}: M_3^\ast \rightarrow M_4^\ast$ be the amalgams. Finally, let $M_1' := M_4$, $M_4' := M_4^\ast$, $f_{4, 4'} := f_{4, 4^\ast}$ $f_{1, 1'} := f_{1, 4}$, $f_{2, 1'} := f_{3,4} \circ f_{2, 3}$, $f_{1', 4'} := f_{3^\ast, 4^\ast} \circ f_{4, 3^\ast}$.
\end{proof}

A transitive independence relation with existence and uniqueness still falls short of the forking-like notion we seek: we must require, too, that it satisfy a suitable local character property. We formulate this by requiring that the induced category be accessible:

\begin{defin}\label{stable-def}
  Let $\nf$ be an independence relation.
  \begin{enumerate}
  \item We say that $\nf$ is \emph{right accessible} if it is right transitive and $\cknf$ is accessible (Definition \ref{acc-def}). We say that $\nf$ is \emph{left accessible} if $\nf^d$ is right accessible. We say that $\nf$ is \emph{accessible} if it is both left and right accessible.
  \item We say that $\nf$ is \emph{stable} if it is symmetric, transitive, accessible, has existence, and has uniqueness.
  \end{enumerate}
\end{defin}
\begin{remark}
  In several examples of interest, $\cknf$ will actually be an AEC. This is the case, for example, when $\ck$ is the category of models of a stable first-order theory or when $\K$ has an almost fully good independence relation in the sense of \cite[A.2]{ap-universal-apal}.
\end{remark}

\begin{remark}\label{acc2}
Let $\nf$ be a right monotonic, right transitive, independence relation with existence. Assume that $\cknf$ is $\lambda$-accessible. Then the embedding of $\cknf$ into $\ck^2$ preserves $\lambda$-directed colimits. Indeed, consider a $\lambda$-directed diagram $D:\cd\to\cknf$. Let $(u_1^d,u_2^d):Dd\to g$ be its colimit in $\ck^2$ and $(\bar{u}_1^d,\bar{u}_2^d):Dd\to \bar{g}$ its colimit in $\cknf$. 
We get a unique morphism $(t_1,t_2):g\to\bar{g}$ in $\ck^2$. Since the outer rectangle of 
$$  
  \xymatrix@=3pc{
    N_i \ar[r]^{u_2^d} & N \ar[r]^{t_2}  & \bar{N} \\
    M_i \ar [u]^{Dd} \ar [r]_{u_1^d} & M \ar[u]_{g} \ar[r]_{t_1} & \bar{M} \ar[u]_{\bar{g}}
  }
  $$
is independent for each $d\in\cd$, \ref{descent} implies that the left squares are independent for each $d\in\cd$. It is easy to check that $t_1$ is an isomorphism, so by Lemma \ref{fd-rmk}, the right square is also independent. Thus $g$ is a colimit of $D$ in $\cknf$.
\end{remark}

While we have made no requirements on the underlying category $\ck$, accessibility of $\cknf$ implies accessibility of $\ck$:

\begin{lem}\label{acc}
  If $\nf$ is right monotonic, right transitive, right accessible and has existence, then $\ck$ is accessible.
\end{lem}
\begin{proof}
Assume that $\cknf$ is $\lambda$-accessible. We begin by proving that $\ck$ has $\lambda$-directed colimits. 
Let $\seq{f_{i, j} : M_i \rightarrow M_j \mid i, j \in I}$ be a $\lambda$-directed diagram in $\ck$. We know that $\cknf$ has 
$\lambda$-directed colimits so (identifying $\ck$ with its copy in $\cknf$, see Remark \ref{full-subcat-rmk}), let 
$\seq{(f_i^1, f_i^2) : (\id_{M_i}:M_i\to M_i) \rightarrow (g: N_1 \rightarrow N_2) \mid i \in I}$ be a colimit in $\cknf$.

We first prove that $g$ is an isomorphism. Indeed, $\seq{(f^1_i,f^1_i): \id_{M_i} \rightarrow \id_{N_1} \mid i \in I}$ is a cocone in $\cknf$. Thus there is a unique morphism $(h_1,h_2): g \rightarrow \id_{N_1}$ such that for all $i \in I$, $h_1f^1_i=f^1_i$ and $h_2f^2_i=f^1_i$. Since $(\id_{N_1},g): \id_{N_1} \rightarrow g$ is a morphism in $\cknf$, for all $i \in I$, $h_1f^1_i=f^1_i$, and $gh_2f^2_i=gf^1_i=f^2_i$, we get that $h_1=\id_{N_1}$ and $gh_2=\id_{N_2}$. Since $h_2g=h_1=\id_{N_1}$, $h_2$ is the inverse of $g$, as desired.

We now claim that $\seq{f_i^2 \mid i \in I}$ is a colimit of the diagram in $\ck$. Let $\seq{g_i : M_i \rightarrow N \mid i \in I}$ be 
a cocone. Then by Lemma \ref{fd-rmk}, $\seq{(g_i, g_i):\id_{M_i}\to\id_N \mid i \in I}$ is a cocone in $\cknf$. Therefore there must exist a unique morphism $(h_1, h_2)$ such that $h_\ell f_i^\ell=g_i$ for $\ell = 1,2$ and $i \in I$. In particular, $h_2 f_i^2=g_i$. Moreover, $h_2$ is unique (in $\ck$) with this property: if $hf_i^2= g_i$ for all $i \in I$, then $hgf_i^1=g_i$ and, by Lemma \ref{fd-rmk} (recalling that, by the preceding paragraph, $g$ is an isomorphism), $(hg,h)$ is a morphism from $g: N_1 \rightarrow N_2$ to $\id_N$ in $\cknf$. Hence $h=h_2$. This completes the proof that $\ck$ has $\lambda$-directed colimits.

Since $\ck$ has $\lambda$-directed colimits, so does $\ck^2$. By Remark \ref{acc2}, $\lambda$-directed colimits are the same in $\ck^2$ and $\cknf$.  It is easy to verify, then, that $\ck^2$, and hence also $\ck$, are $\lambda$-accessible.
\end{proof}

\begin{remark}\label{acc1}
  We have just shown that the embedding of $\cknf$ into $\ck$ is accessible. In the other direction, the proof above shows that the category $\ck$ is accessibly embedded (see \cite[2.35]{adamek-rosicky}) into $\cknf$.
\end{remark}

Let us illustrate what a stable independence relation looks like by considering accessible categories with weak \emph{polycolimits}:

\begin{defin}\label{poly-def}
  A \emph{polyinitial} object is a set $\ci$ of objects of a category $\ck$ such that for every object $M$ in $\ck$:

\begin{enumerate}
\item There is a unique $i \in\ci$ having a morphism $i \to M$.
\item For each $i \in \ci$, given $f,g:i \to M$, there is a unique (isomorphism) $h:i \to i$ with $fh=g$.
\end{enumerate}

The \emph{polycolimit} of a diagram $D$ in a category $\ck$ is a polyinitial object in the category of cones on $D$, i.e. a set of cones such that for any cone on $D$, there will be an induced map from exactly one of the members of the set, and this map will be unique up to unique isomorphism, in the sense of \ref{poly-def}(2). The \emph{weak polycolimit} of $D$ is defined similarly, except that we waive the uniqueness requirement on the map.
\end{defin}

For example, the algebraic closures of the prime fields form a polyinitial object in the category of algebraically closed fields. 

We recall, too, the related concept of a multicolimit, again derived from the notion of a multiinitial set of objects: a multiinitial set of objects in a category $\ck$ is given by replacing (2) above---uniqueness of the morphism $i\to M$ up to isomorphism---by actual uniqueness.  That is, for each $M\in\ck$, there is a unique $i\in\mathcal I$ and a unique $i\to M$.

In the category of fields, for example, the prime fields form a multiinitial object.

We say, incidentally, that a category is locally $\lambda$-polypresentable (respectively, locally $\lambda$-multipresentable) if it is $\lambda$-accessible and has all polycolimits (respectively, multicolimits).  Note that locally $\lambda$-presentable implies locally $\lambda$-multipresentable implies locally $\lambda$-polypresentable implies $\lambda$-accessible.

\begin{remark}\label{multi} \
  \begin{enumerate}
    \item\label{multi-1} Let $\nf$ be a stable independence relation in an accessible category $\ck$ with weak polycolimits and all morphisms monomorphisms (for example in a $\mu$-AEC admitting intersections, see \cite[2.4, 5.7]{multipres-pams}). Consider a span $s=(f_1,f_2)$ and its weak polypushout $(p_{i,1},p_{i,2})$, $i\in I_s$ (note that in general $I_s$ may be empty but here it is not: the existence property of $\nf$ implies amalgamation). This means that we have commutative squares
$$
\xymatrix@=3pc{
M_1 \ar[r]^{p_{i,1}} & P_i \\
M_0 \ar [u]^{f_1} \ar [r]_{f_2} &
M_2 \ar[u]_{p_{i,2}}
}
$$
such that for any commutative square
$$
\xymatrix@=3pc{
M_1 \ar[r]^{g_1} & M \\
M_0 \ar [u]^{f_1} \ar [r]_{f_2} &
M_2 \ar[u]_{g_2}
}
$$
there exist a unique $i\in I_s$ and (not necessarily unique) $t:P_i\to M$ such that $tp_{i,1}=g_1$ and $tp_{i,2}=g_2$. By closure under $\sim$ and the uniqueness property of the weak polypushout, there is exactly one $i\in I_s$ such that $\nfs{M_0}{M_1}{M_2}{P_i}$.
By monotonicity, these choices are coherent in the sense that given a morphism of spans $(\id_{M_0},h_1,h_2):s\to s'$,
we get an induced morphism $P_i\to P_{i'}$ where $i$ is chosen from $I_s$ and $i'$ from $I_{s'}$. For a general object $M$, the characterization of the relation $\nfs{M_0}{M_1}{M_2}{M}$ is by closure under $\sim$: the relation holds if and only if the induced map from the weak polypushout is $P_i\to M$ for the unique $P_i$ such that $\nfs{M_0}{M_1}{M_2}{P_i}$. 

On the other hand, a relation $\nf$ given by a coherent choice of weak polypushouts satisfies all axioms of stable independence with the possible exception of accessibility (and existence, if $\ck$ does not have amalgamation). We will see (Theorem \ref{canon-thm}) that in a category with chain bounds there is at most one coherent choice of weak polypushouts giving a stable independence relation.
\item We will see in the next two sections that for any coregular locally presentable category $\ck$, pushouts of regular monomorphisms form a coherent choice of multipushouts in $\ck_{reg}$ (the subcategory of $\ck$ with the same objects but only regular monomorphisms). These pushouts will in particular be pullback squares in $\ck_{reg}$.
\item One cannot expect that an accessible category with weak polycolimits and all morphisms monomorphisms will have pushouts. In this case,
  we might try to get a stable independence relation by taking all commutative squares
  
$$
\xymatrix@=3pc{
M_1 \ar[r]^{} & M_3 \\
M_0 \ar [u]^{} \ar [r]_{} &
M_2 \ar[u]_{}
}
$$
Without the existence of pushouts, this relation satisfies all axioms of a stable independence relation except the uniqueness axiom.

We note, too, that one must be careful in assuming pushouts, as this will often prove too strong of a restriction on the category $\ck$.  For example, let $\ck$ be an accessible category with multicolimits and all morphisms monomorphisms.  If $\ck$ has pushouts, it must be small.  To see this, consider an object $M$,
the instance $O$ of a multiinitial object with morphism $O\to M$ and the pushout
$$
\xymatrix@=3pc{
M \ar[r]^{} & M\coprod M \\
O \ar [u]^{} \ar [r]_{} &
M \ar[u]_{}
}
$$
Since the codiagonal $M\coprod M\to M$ is a split epimorphism, it is an isomorphism and thus the two coproduct injections
$M\to M\coprod M$ are equal. Consequently, for any $N$ in $\ck$, $M$ has at most one morphism $M\to N$. Thus $\ck$ is thin and, since it is an accessible category, it must therefore be small.
\end{enumerate}
\end{remark}

We end this section by giving several examples and non-examples:

\begin{exams}\label{indep-examples} \
\begin{enumerate}
  \item\label{indep-ex-1} If $\K$ is the AEC of all sets (ordered with substructure), then $\nfs{M_0}{M_1}{M_2}{M_3}$ if and only if $M_0 \subseteq M_\ell \subseteq M_3$, $\ell = 1,2$, and $M_1 \cap M_2 = M_0$ is a stable independence relation.
  \item\label{indep-ex-2} For $F$ a field, if $\K$ is the AEC of all $F$-vector spaces (ordered with subspace), then $\nfs{M_0}{M_1}{M_2}{M_3}$ if and only if $M_0 \subseteq M_\ell \subseteq M_3$, $\ell = 1,2$, and $M_1 \cap M_2 = M_0$ is a stable independence relation. We will see that this example (along with (1) above) is an instance of Theorem \ref{indep}, which gives a general condition on when pullback induces a stable independence notion.

  \item\label{acf-example} If $\K$ is the AEC of all algebraically closed fields of characteristic $p$ (for $p$ a fixed prime or 0) ordered by subfield, define  $\nfs{M_0}{M_1}{M_2}{M_3}$ if and only if $M_0 \lea M_\ell \lea M_3$ and for any finite $A \subseteq U M_1$, and any $a \in U M_1$, the transcendence degree of $a$ over $A M_2$ is the same as the transcendence degree of $a$ over $AM_0$.  This is a stable independence notion. Note however that $M_0 \lea M_\ell \lea M_3$ and $M_1 \cap M_2 = M_0$ does \emph{not} imply $\nfs{M_0}{M_1}{M_2}{M_3}$. This is because the pregeometry induced by algebraic closure is not modular. Nevertheless, $\K$ has weak polycolimits, hence we are in the setup described by Remark \ref{multi}(\ref{multi-1}).
  \item\label{field-example} Let $\K$ be the AEC of all (not necessarily algebraically closed) fields of characteristic $p$ (for $p$ a fixed prime or 0) ordered by elementary substructure. Then it is well-known that $\K$ does not have any stable independence notion (indeed, ordered fields belong to $\K$, so $T$ is not a stable theory). However, the subclass of $\aleph_0$-saturated models of $\K$ is nothing but the AEC of algebraically closed fields of characteristic $p$ ordered by subfield (by the model-completeness of algebraically closed fields), which \emph{does} have a stable independence notion. 
  \item\label{fo-example} Let $T$ be a stable first-order theory. Write $\nfs{A}{B}{C}{M}$ if and only if $A, B, C \subseteq U M$, $M \models T$, and $\tp (\bb / AC; M)$ does not fork over $A$ (in the original sense of Shelah, see \cite[III.1.4]{shelahfobook}), where $\bb$ is any enumeration of $B$ and $\tp (\bb / AC; M)$ is the set of first-order formulas with parameters from $AC$ satisfied by $\bb$ in $M$. There is a certain expansion $T^{eq}$ of $T$ which ``eliminates imaginaries'' in the sense that for every definable equivalence relation $E$ there is a definable function $F_E$ sending two $E$-equivalent elements to the same object (i.e.\ the equivalence classes of $E$ are also elements), see \cite[\S III.6]{shelahfobook}.\footnote{We note that the passage from $T$ to $T^{eq}$ has a category-theoretic generalization, developed independently (from Shelah) by Makkai and Reyes; that is, the passage from a logical category $\mathbb T$ to its {\it pretopos completion}, \cite{makkai-reyes}.  That these ideas are connected is due to Makkai, see e.g. \cite{harnik-teq}.} Expanding to $T^{eq}$ leads to an equivalent category of models, so without loss of generality $T = T^{eq}$. Then it is known that $\nfs{A}{B}{C}{M}$ if and only if $\nfs{\acl (A)}{\acl (AB)}{\acl (AC)}{M}$, where $\acl (A)$ denotes the set of elements $b$ in $M$ for which there is a formula $\phi (x)$ with parameters from $A$ with only finitely many solutions, one of which is $b$. Thus it suffices to consider forking over algebraically closed sets. Let $\K$ be the AEC of algebraically closed sets in $T$, ordered by being a substructure (see \cite[V.B.2.1]{shelahaecbook2}). Then $\nf$ induces a stable independence relation on $\K$.
  \item\label{graph-example} Let $\K$ be the AEC of all (undirected) graphs, ordered by being a full subgraph. Define $\nfs{M_0}{M_1}{M_2}{M_3}$ to hold if and only if $M_0 \subseteq M_\ell \subseteq M_3$, $\ell = 1,2$, $M_1 \cap M_2 = M_0$, and there are no edges of $M_3$ going from $M_1 \backslash M_0$ to $M_2 \backslash M_0$. Then $\nf$ satisfies all the axioms of a stable independence relation except accessibility. One can define a different independence notion (that will similarly satisfy all the axioms except accessibility) similarly by requiring instead that \emph{all} possible cross edges are present between $M_1 \backslash M_0$ and $M_2 \backslash M_0$. This is \cite[II.6.4]{shelahaecbook}, see also \cite[4.15]{bgkv-apal}. We will see (Example \ref{locally-finite-example}) that when considering only graphs whose vertices have bounded degree, $\nf$ is stable and becomes the only independence notion.
  \item\label{graph-example-2} Let $\K$ be the AEC of all graphs, as before. One can modify $\nf$ so that it satisfies existence, uniqueness, and monotonicity but not transitivity: say $\nfs{M_0}{M_1}{M_2}{M_3}$ holds if and only if $M_0 \subseteq M_\ell \subseteq M_3$, $\ell = 1,2$, and:
    \begin{enumerate}
    \item If $M_0$ is a finite graph, then no cross edges are present between $M_1 \backslash M_0$ and $M_2 \backslash M_0$.
    \item If $M_0$ is infinite, all possible cross edges are present between $M_1 \backslash M_0$ and $M_2 \backslash M_0$.
    \end{enumerate}

    One can similarly construct an independence relation with existence and uniqueness that fails to be monotonic (use the isomorphism type of $M_2$ to decide whether all or no cross edges should be included). It is also easy to construct examples that have existence, but not uniqueness, or uniqueness but not existence, see \cite[4.16]{bgkv-apal}.
  \item\label{categ-example} In all the examples given so far, the morphisms were monomorphisms. This need not be the case, however: in any accessible category with pushouts, the class of \emph{all} squares form a stable independence notion. In this sense, stable independence generalizes pushouts. We explore this theme further in the next two sections.
\end{enumerate}
\end{exams}

\section{Coregular categories and effective unions}\label{coreg-sec}

Recall that a monomorphism (respectively, epimorphism) in a category is \emph{regular} if it is an equalizer (respectively, coequalizer) of a pair of morphisms (see \cite[7.56]{joy-of-cats}). 

\begin{notation}\label{kreg-not} Let $\ck$ be a category.  We denote by $\ck_{reg}$ the full subcategory of $\ck$ whose morphisms are precisely the regular morphisms of $\ck$.\end{notation}

\begin{remark}\label{pullback-regular}
  In $\mu$-AECs with disjoint amalgamation (and more generally in any category where spans can be completed to pullback squares), every monomorphism is regular.
\end{remark}

A category $\ck$ is called \emph{regular} if it has finite limits, coequalizers of kernel pairs and regular epimorphisms are stable under pullbacks (see \cite{barr-exact}). Coregularity is the dual of this notion.  So, in particular, a locally presentable category $\ck$ is coregular if and only if regular monomorphisms are stable under pushouts. In this case, we have the factorization system (epimorphism, regular monomorphism) in $\ck$. This implies that $\ck_{reg}$ is quite a well-behaved category:

\begin{lemma}\label{lp} Let $\ck$ be a coregular locally $\lambda$-presentable category. Then $\ck_{reg}$ is locally $\lambda$-multipresentable.
\end{lemma}
\begin{proof} Following \cite[1.62]{adamek-rosicky}, $\ck_{reg}$ is closed under $\lambda$-directed
colimits in $\ck$. Since $\lambda$-pure monomorphisms are regular (\cite[2.31]{adamek-rosicky}), $\ck_{reg}$ is accessible by \cite[2.34]{adamek-rosicky}.
We will prove that each object $K$ of $\ck$ is a $\lambda$-directed colimit of $\lambda$-presentable objects
in $\ck_{reg}$. Let $a_i:A_i\to K$ be a $\lambda$-directed colimit of $\lambda$-presentable objects $A_i$ in $\ck$. Since $\ck$ has (epimorphism, regular monomorphism)-
factorizations, we take $a_i=a_i''a_i'$ to be a (epimorphism, regular monomorphism) factorization of $a_i$. Then $a_i'':B_i\to K$ is
a $\lambda$-directed colimit of regular monomorphisms. Analogously as in \cite[1.69]{adamek-rosicky}(i), objects $B_i$ are $\lambda$-presentable in $\ck_{reg}$.
Thus $\ck_{reg}$ is $\lambda$-accessible.

$\ck_{reg}$ is clearly closed under equalizers in $\ck$. Let $p_i:P\to M_i$ be a wide pullback of regular monomorphisms
$f_i:M_i\to M$. Since regular monomorphisms form the right part of a factorization system, the composition $f_ip_i$ is a regular monomorphism. For the same reason,
the $p_i$ are regular monomorphisms. Thus $\ck_{reg}$ has connected colimits and, since it is $\lambda$-accessible, it is locally $\lambda$-multipresentable.
\end{proof}


We will use the following important fact:

\begin{fact}[\cite{ringel}]\label{ringel-fact}
In a coregular locally presentable category, pushouts of regular monomorphisms are pullbacks.  

In fact, this is valid in any category where pushouts of regular monomorphisms are regular monomorphisms.
\end{fact}

Note, however, that pushouts of regular monomorphisms in $\ck$ need not be pushouts in $\ck_{reg}$---the induced map from the pushout will not even be a monomorphism, in general, let alone a regular monomorphism.  To ask that this be the case is precisely to ask that $\ck$ have {\it effective unions}. The existence of effective unions is an exactness property (introduced by Barr \cite{effective-unions}) satisfied by both Grothendieck abelian categories and Grothendieck toposes and which is strong enough to ensure the existence of enough injective objects---see Remark~\ref{cof-rmk} below.

\begin{defin}
  A locally presentable category $\ck$ \emph{has effective unions} if whenever we have a pullback
  
$$
\xymatrix@=3pc{
M_1 \ar[r]^{} & M_3 \\
M_0 \ar [u]^{} \ar [r]_{} &
M_2 \ar[u]_{}
}
$$

in $\ck_{reg}$, and the pushout

$$
\xymatrix@=3pc{
M_1 \ar[r]^{} & P \\
M_0 \ar [u]^{} \ar [r]_{} &
M_2 \ar[u]_{}
}
$$

in $\ck$, the induced morphism $P\to M_3$ is a regular monomorphism.
\end{defin}

\begin{remark}
If $\ck$ is a coregular locally presentable category then any morphism in $\ck_{reg}$ is a regular monomorphism. This follows from the fact that pushout squares in $\ck$ will be pullbacks (Fact \ref{ringel-fact}) and Remark \ref{pullback-regular}.
\end{remark}

\begin{remark}\label{cof-rmk}
  Effective unions are closely tied to the existence of enough injectives (see e.g.\ Chapter 4 in \cite{adamek-rosicky}): let $\ck$ be a coregular locally presentable category such that regular monomorphisms are closed under directed colimits 
  (in the sense of \cite[2.1(2)]{uq-cellular-v2-toappear}). Then $\ck_{reg}$ has directed colimits. Moreover, if $\ck$ has effective unions then regular monomorphisms are cofibrantly generated (in the sense given at the beginning of \cite[p.~7]{uq-cellular-v2-toappear}; the proof is the same as that of \cite[1.12]{beke-sheafifiable}). In particular, there is a set $\cs$ of regular monomorphisms such that an object injective with respect to any $s\in\cs$ is injective to all regular monomorphisms. Consequently, these injective objects form an accessible category (see \cite[4.7]{adamek-rosicky}).
\end{remark}

\begin{exams}\label{effective-union-examples} \
  \begin{enumerate}
  \item\label{effective-1} Both Grothendieck toposes and Grothendieck abelian  categories are locally presentable coregular categories having effective unions and regular monomorphisms closed under directed colimits (see \cite{effective-unions}). In particular, they include categories of $R$-modules and presheaf categories $\Set^{\cx}$. The latter include the category of multigraphs.
  \item The category $\Gra$ of graphs is locally presentable and coregular but it does not have effective unions. It suffices
to take the pullback of two vertices of an edge, where the pushout consists two vertices without any edge. Thus the embedding 
of this pushout to the edge is not regular (although, we note, it would be regular in multigraphs). See also Example \ref{indep-examples}(\ref{graph-example}).
  \item The category of groups and the category of Boolean algebras are locally presentable and coregular. Regular monomorphisms coincide with monomorphisms and are closed under directed colimits. But groups do not have enough injectives, and the injectives in Boolean algebras are complete Boolean algebras: these do not form an accessible category. Thus neither groups nor Boolean algebras have effective unions (see Remark \ref{cof-rmk}).
  \item  The category $\Ban$ of Banach spaces (with linear contractions) is locally presentable and coregular. The regular monomorphisms are isometries and are closed under directed colimits. However $\Ban$ does not have effective unions: if it did, by Remark \ref{cof-rmk}, regular monomorphisms would be cofibrantly generated. This contradicts \cite[3.1(2)]{uq-cellular-v2-toappear}.  We provide an alternate proof of this fact in Remark~\ref{banachord-rmk}, exploiting the connection between stable independence and the failure of the order property.
   \item The category $\Hilb$ of Hilbert spaces (with linear isometries) is accessible. Recall that Hilbert spaces are precisely the Banach spaces whose norm satisfies the parallelogram identity.  By the note above, $\Hilb$ is a (full) subcategory of $\Ban_{reg}$.  Pullbacks in $\Hilb$ exist and are calculated as
for Banach spaces.  If $f:V\to W$ is a morphism 
in $\Hilb$ then $W\cong V\oplus V^\perp$ where $V^\perp$ is the orthogonal complement of $f(V)$ in $W$. Thus a typical span of maps consists of 
$f:V\to V\oplus W_1$ and $g:V\to V\oplus W_2$, with pushout $V\oplus W_1\oplus W_2$. Since a general pullback has the form
$$
\xymatrix@=3pc{
V\oplus W_1 \ar[r]^{} & V\oplus W_1\oplus W_2\oplus W_3 \\
V \ar [u]^{} \ar [r]_{} &
V\oplus W_2 \ar[u]_{}
}
$$
and the induced map $V\oplus W_1\oplus W_2\to V\oplus W_1\oplus W_2\oplus W_3$ is clearly an equalizer, $\Hilb$ has effective unions.
\end{enumerate}
\end{exams}

\section{Stable independence and effective unions}\label{effect-sec}

\begin{thm}\label{indep}
Let $\ck$ be a coregular locally presentable category. If $\ck$ has effective unions, then $\ck_{reg}$ has a stable independence relation.
\end{thm}
\begin{proof}
Put $\nfs{M_0}{M_1}{M_2}{M_3}$ if and only if the square
$$
\xymatrix@=3pc{
M_1 \ar[r]^{} & M_3 \\
M_0 \ar [u]^{} \ar [r]_{} &
M_2 \ar[u]_{}
}
$$
is a pullback. We check that this satisfies all the axioms of stable independence. If $\nfs{M_0}{M_1}{M_2}{M_3}$ and $M_3\to M'_3$ then
$\nfs{M_0}{M_1}{M_2}{M'_3}$. Given a commutative square
$$
\xymatrix@=3pc{
M_1 \ar[r]^{} & M''_3 \\
M_0 \ar [u]^{} \ar [r]_{} &
M_2 \ar[u]_{}
}
$$
and $M''_3\to M_3$ such that $\nfs{M_0}{M_1}{M_2}{M_3}$, then $\nfs{M_0}{M_1}{M_2}{M''_3}$. This shows that $\nf$ is closed under equivalence of amalgams.  We have used here that if $gf$ is a regular monomorphism, then $f$ is a regular monomorphism (see \cite[2.1.4]{freyd-kelly} -- recall that in any coregular locally presentable category we have the factorization system (epimorphism, regular monomorphism)). 

Given a span $M_1\leftarrow M_0\to M_2$, we form the pushout
$$
\xymatrix@=3pc{
M_1 \ar[r]^{} & M_3 \\
M_0 \ar [u]^{} \ar [r]_{} &
M_2 \ar[u]_{}
}
$$
Following Fact \ref{ringel-fact}, this square is a pullback. Thus $\nfs{M_0}{M_1}{M_2}{M_3}$, which yields the existence property.

In order to prove the uniqueness property, consider $\nfs{M_0}{M_1}{M_2}{M'_3}$ and $\nfs{M_0}{M_1}{M_2}{M''_3}$ with the same
span $M_1\leftarrow M_0\to M_2$. Form the pushout 
$$
\xymatrix@=3pc{
M_1 \ar[r]^{} & P \\
M_0 \ar [u]^{} \ar [r]_{} &
M_2 \ar[u]_{}
}
$$
and take the induced regular monomorphisms $P\to M'_3$ and $P\to M''_3$. Then the pushout
$$
\xymatrix@=3pc{
M'_3 \ar[r]^{} & M_3 \\
P \ar [u]^{} \ar [r]_{} &
M''_3 \ar[u]_{}
}
$$

amalgamates the starting pullbacks.

The transitivity and symmetry properties are also easy to check. We can prove accessibility directly or use the soon to be proven Theorem \ref{accessible-charact}: local character holds since the nonforking base is simply given by the pullback, and the witness property is also a straightforward consequence of the definition.
\end{proof} 

\begin{remark}\label{indep1} \
  \begin{enumerate}
  	\item We note that the assumption of local presentability is largely a matter of convenience, and the argument will go through under weaker assumptions.  In particular, the category need only be accessible with pullbacks and pushouts of monomorphisms.
    \item Effective unions were only needed for the uniqueness property. On the other hand, if the independence relation defined by being a pullback square satisfies uniqueness, then the category has effective unions. Consider
$\nfs{M_0}{M_1}{M_2}{M_3}$ and form the pushout
$$
\xymatrix@=3pc{
M_1 \ar[r]^{} & P \\
M_0 \ar [u]^{} \ar [r]_{} &
M_2 \ar[u]_{}
}
$$
Since $\nfs{M_0}{M_1}{M_2}{P}$ (by \ref{ringel-fact}), we have an amalgam $M_3\to M\leftarrow P$ whose right leg is the composition
$P\to M_3\to M$. Thus $P\to M_3$ is a regular monomorphism.
\item The converse of Theorem \ref{indep} fails: Example \ref{locally-finite-example} shows that the category $\ck$ of graphs whose vertices have degree at most $k$ (for fixed $k < \omega$) is coregular, locally presentable, and has a stable independence relation in $\ck_{reg}$, but does \emph{not} have effective unions.
\item We show in Theorem \ref{indep-pullback} that in any $\mu$-AEC, independent squares over sufficiently saturated models are pullbacks.
\item Let $\ck$ be a coregular locally presentable category having effective unions. Then in $\ck_{reg}$, given a span $(f_1,f_2)$,
 exactly one instance of a multipushout of $f_1$ and $f_2$ is a pullback. This follows from Remark \ref{multi} and Theorem \ref{indep}.
  \end{enumerate}
\end{remark}

\begin{exams}\label{indep-examples1}
Both Grothendieck toposes and Grothendieck  abelian categories have stable independence relations. This follows from Theorem \ref{indep}
and Example \ref{effective-union-examples}(\ref{effective-1}) and subsumes \ref{indep-examples}(\ref{indep-ex-1}) and (\ref{indep-ex-2}). In fact, (\ref{indep-ex-2}) is valid for $R$-modules in general. This was known already for complete first-order theories of modules (see \cite{prest}): we have produced an alternate proof of this fact. Similarly, the category of Hilbert spaces (with linear isometries) has a stable independence notion---as $\Hilb$ is accessible but not locally presentable, we must actually invoke the weakening mentioned in \ref{indep}(1) above. This was also known from the model-theoretic analysis of the corresponding continuous first-order theory, see for example \cite{iovino-forking-banach}.
\end{exams}

In the rest of this section, we show that even without effective unions, we can still define an independence notion that satisfies all the axioms of stable independence except accessibility.

\begin{defin}\label{nf1-def}
  Let $\ck$ be a coregular locally presentable category.  Define $\nf^\ast$ by taking commutative squares in $\ck_{reg}$
  
$$
\xymatrix@=3pc{
M_1 \ar[r]^{} & M_3 \\
M_0 \ar [u]^{} \ar [r]_{} &
M_2 \ar[u]_{}
}
$$
such that the induced morphism from the pushout
$$
\xymatrix@=3pc{
M_1 \ar[r]^{} & P \\
M_0 \ar [u]^{} \ar [r]_{} &
M_2 \ar[u]_{}
}
$$
is a regular monomorphism.
\end{defin}
\begin{remark}
These are precisely \textit{effective pullback squares}

$$
\xymatrix@=3pc{
M_1 \ar[r]^{} & M_3 \\
M_0 \ar [u]^{} \ar [r]_{} &
M_2 \ar[u]_{}
}
$$

in the sense that the unique morphism $P\to M_3$ is a regular monomorphism. In $\Gra$, $\nf^\ast$ is the relation $\nf$ mentioned in Example \ref{indep-examples}(\ref{graph-example}).
\end{remark}

If $\ck$ has effective unions then $\nf^\ast=\nf$ where $\nf$ is the relation from Theorem \ref{indep}. This uses the fact that pushouts in $\ck$ are pullbacks. 

Moreover, by the next result, together with Theorem \ref{canon-thm}, if $\ck_{reg}$ has chain bounds then whenever $\nf$ is \textit{any} stable independence relation on $\ck_{reg}$ we have that $\nf = \nf^\ast$.  

\begin{thm}\label{nf1-thm}
  Let $\ck$ be a coregular locally presentable categories and let $\nf^\ast$ be as in Definition \ref{nf1-def}.

  Then:
  \begin{enumerate}
  \item $\nf^\ast\subseteq\nf$ (where $\nf$ is from Definition \ref{indep}).
  \item $\nf^\ast$ is invariant, monotonic, symmetric, transitive, has existence and uniqueness.
  \item If $\ck$ is locally $\lambda$-presentable then the category $\cl$ of $\nf^\ast$-independent squares in $\ck_{reg}$ (from Definition \ref{k-nf-def}) is closed under $\lambda$-directed colimits in $\ck^2$.
  \item If $\ck$ is locally $\lambda$-presentable, $\ck^\square$ is the category of commutative squares in $\ck$, and $\ck^\boxdot$ is the category of effective pullback squares in $\ck$, then $\ck^\boxdot$ is $\lambda$-accessible with $\lambda$-presentable objects being squares of $\lambda$-presentable objects in $\ck$.
  \end{enumerate}
\end{thm}
\begin{proof} \
  \begin{enumerate}
  \item Straightforward.
  \item The properties of $\nf^\ast$ are proven as in the proof of Theorem \ref{indep}, except for transitivity. To prove transitivity, consider:
  
$$  
  \xymatrix@=3pc{
    M_1 \ar[r]^{} & M_3 \ar[r]^{}  & M_5 \\
    M_0 \ar [u]^{} \ar [r]_{} & M_2 \ar[u]_{} \ar[r]_{} & M_4 \ar[u]_{}
  }
  $$
where both squares are effective pullbacks. We have to show that the outer rectangle is an effective pullback.
Thus we have to show that the induced morphism $p:P\to M_5$ from the pushout
$$
\xymatrix@=3pc{
M_1 \ar[r]^{} & P \\
M_0 \ar [u]^{} \ar [r]_{} &
M_4 \ar[u]_{}
}
$$  
is a regular monomorphism. This pushout is a composition of pushouts
$$  
  \xymatrix@=3pc{
    M_1 \ar[r]^{} & Q \ar[r]^{}  & P \\
    M_0 \ar [u]^{} \ar [r]_{} & M_2 \ar[u]_{} \ar[r]_{} & M_4 \ar[u]_{}
  }
  $$
Consider the pushout
 $$
\xymatrix@=3pc{
M_3 \ar[r]^{} & P' \\
Q \ar [u]^{q} \ar [r]_{} &
P \ar[u]_{\bar{q}}
}
$$  
where $q:Q\to M_3$ is the induced morphism. Since the left pullback above is effective, $q$ is a regular monomorphism and thus
$\bar{q}$ is a regular monomorphism. Since
$$
\xymatrix@=3pc{
M_3 \ar[r]^{} & P' \\
M_2 \ar [u]^{} \ar [r]_{} &
M_4 \ar[u]_{}
}
$$ 
is a pushout and the right pullback above is effective, the induced morphism $p':P'\to M_5$ is a regular monomorphism. Thus
$p=p'\bar{q}$ is a regular monomorphism. Conversely, if the outer rectangle and right square in the proof above are effective pullbacks then the left square is an effective pullback. This follows directly from monotonicity. In fact, the composition of $q$ with $M_3\to M_5$ is a regular monomorphism as the composition of $Q\to P$ with $p$. Thus $q$ is a regular monomorphism (see \cite[2.1.4]{freyd-kelly}).
\item Let $D:\cd\to \cl$ be a $\lambda$-directed diagram where $Dd$ is $f_d:M_d\to N_d$. Let $f:M\to N$
be its colimit in $\ck^2$. For each $d\in\cd$, the square
$$
\xymatrix@=3pc{
M \ar[r]^{f} & N \\
M_d \ar [u]^{} \ar [r]_{f_d} &
N_d \ar[u]_{}
}
$$
is a pullback because it is a $\lambda$-directed colimit of pullbacks
$$
\xymatrix@=3pc{
M_d' \ar[r]^{f_{d'}} & N_d' \\
M_d \ar [u]^{} \ar [r]_{f_d} &
N_d \ar[u]_{}
}
$$
and pullbacks commute with $\lambda$-directed colimits in $\ck$ (see \cite{adamek-rosicky} 1.59). Analogously, the pushout
$$
\xymatrix@=3pc{
M \ar[r]^{g} & P \\
M_d \ar [u]^{} \ar [r]_{f_d} &
N_d \ar[u]_{}
}
$$
is a $\lambda$-directed colimit of pushouts
$$
\xymatrix@=3pc{
M_d' \ar[r]^{g_{d'}} & P_{d'} \\
M_d \ar [u]^{} \ar [r]_{f_d} &
N_d \ar[u]_{}
}
$$
Thus the induced morphism $p:P\to N$ is a $\lambda$-directed colimit of induced morphisms $p_{d'}:P_{d'}\to N_{d'}$. Hence $p$ is a regular monomorphism.
\item The category $\ck^\boxdot$ is accessible and accessibly embedded into the category $\ck^\square$ because it is given by the following sketch (see \cite{adamek-rosicky} 2.60)  
$$
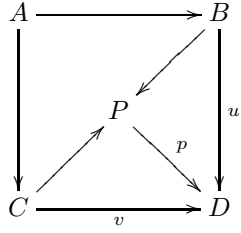
\xymatrix{
A\ar [rr]^{} \ar [dd]_{} && B \ar [dl]_{}
\ar[dd]^{u}\\
& {P}\ar [dr]^{p} &\\
C \ar [rr]_{v}  \ar [ur]^{}&& D
}
$$
where the outer rectangle is a pullback, the inner quadrangle is a pushout and $u,v,p$ are regular monomorphisms. There is a regular cardinal $\mu\triangleright\lambda$ such that $\ck^\boxdot$ is $\mu$-accessible and the inclusion $G:\ck^\boxdot\to\ck^\square$ preserves $\mu$-directed colimits and $\mu$-presentable objects.
\end{enumerate}
\end{proof}

We return to the relation $\nf^{\ast}$ in Example~\ref{nfastwitness}, as it provides a useful example of the \emph{witness property} (Definition~\ref{wp-def}).

\begin{example}\label{locally-finite-example}
  Let $\kappa \ge 2$ be a (possibly finite) cardinal. We call a graph \emph{$\kappa$-local} if all its vertices have degree strictly less than $\kappa$. Let $\ck$ be the category of $\kappa$-local graphs (seen as a full subcategory of the category $\Gra$ of graphs), and let $\theta := \kappa^+ + \aleph_0$. Then $\ck_{reg}$ is closed in $\Gra$ under connected limits, pushouts and $\theta$-directed colimits. Thus $\ck_{reg}$ is a locally $\theta$-multipresentable category. Moreover if $\kappa < \aleph_0$, then $\ck$ is coregular and locally finitely presentable. On the other hand, like $\Gra$, $\ck$ does \emph{not} have effective unions. Nevertheless, we show that it has a stable independence relation. Let $\nf^\ast$ be the independence relation on $\ck_{reg}$ given by Definition \ref{nf1-def} (in this case, this will be as given by Example \ref{indep-examples}(\ref{graph-example}): two graphs are independent over the base if there are no cross edges except over the base). By the proof of Theorem \ref{nf1-thm} (or by checking directly), $\nf^\ast$ satisfies all the properties in the definition of a stable independence notion, except perhaps accessibility. However in this case $(\ck_{reg})_{\NF}$ \emph{is} accessible. Regular monomorphisms in $\ck$ are induced by the relation of being a full subgraph, so let $G \subseteq H$ be $\kappa$-local graphs with $G$ a full subgraph of $H$. Given a set $A \subseteq U H$, there are full subgraphs $G_0 \subseteq H_0$ such that $H_0$ contains $A$, $G_0$ contains $A \cap G$, $|U H_0| \le |A| + \theta$, and the square  
$$  
      \xymatrix@=3pc{
        G \ar@{}\ar[r]{} & H \\
        G_0 \ar [u]^{} \ar [r]_{} &
        H_0 \ar[u]_{}
      }
      $$
      
is an effective pullback (i.e.\ any cross edge between $G$ and $H_0$ inside $H$ is inside $G_0$). It suffices to add to $A \cap G$ all vertices of $G$ connected with a vertex of $A$ by an edge. Thus $\nf^\ast$ is a stable independence relation. 

We consider the canonicity of this relation in Remark~\ref{canon-locally-finite}.
\end{example}

\section{From accessible category to $\mu$-AEC}\label{categ-mu-aec}

In Lemma \ref{lp}, we examined the relationship between a certain accessible category $\ck$ and the category $\ck_{reg}$ obtained by restricting to the regular monomorphisms.  In this section we prove several more results along these lines, both in connection with $\ck_{reg}$ and with $\ck_{mono}$, where the morphisms are taken to be all monomorphisms in the original category $\ck$. The latter is particularly useful in light of Fact \ref{mu-aec-acc}: accessible categories with all morphisms monomorphisms are $\mu$-AECs, and vice versa.  This provides the essential Rosetta stone that allows us to translate the uniqueness and canonicity results derived for $\mu$-AECs below back to the abstract categorical framework. We deduce in particular (using the corresponding result of the third author for universal classes \cite{ap-universal-apal, categ-universal-2-selecta}) that the eventual categoricity conjecture holds for all locally $\aleph_0$-multipresentable categories (Corollary \ref{evcatlocnullpres}). 

\begin{notation} For a category $\ck$, we denote by $\ck_{mono}$ the full subcategory of $\ck$ whose morphisms are precisely the monomorphisms of $\ck$.\end{notation}

 
\begin{lem}\label{mono}
Let $\ck$ be a $\lambda$-accessible category. Then $\ck_{mono}$ is accessible and has $\lambda$-directed colimits.
\end{lem}
\begin{proof}
Let $\ca$ be a (representative) full subcategory of $\ck$ consisting of $\lambda$-presentable objects and
$$
E:\ck\to \Set^{\ca^{\op}}
$$
be the canonical functor (see \cite[1.25]{adamek-rosicky}). Then $E$ is a full embedding preserving $\lambda$-directed colimits (see \cite[1.26]{adamek-rosicky}). Clearly, $E$ preserves monomorphisms too. Thus monomorphisms in $\ck$ are stable under $\lambda$-directed colimits
because this is true in $\Set^{\ca^{\op}}$ (see \cite[1.60]{adamek-rosicky}). That is, given $\lambda$-directed diagrams $D,D':\cd\to\ck$
and a natural monotransformation $\delta:D\to D'$, then $\colim\delta:\colim D\to\colim D'$ is a monomorphism. Consequently,
$\ck_{mono}$ is closed under $\lambda$-directed colimits in $\ck$. That is, given a $\lambda$-directed colimit of monomorphisms
in $\ck$, (i) the colimit cocone consists of monomorphisms, and (ii) for every cocone of monomorphisms the factoring morphism 
is a monomorphism.

Following \cite[2.34]{adamek-rosicky}, there is a regular cardinal $\mu\triangleright\lambda$ such that each object of $\ck$ is a $\mu$-directed
colimit of $\mu$-presentable $\lambda$-pure subobjects of $\ck$. Thus $\ck_{mono}$ is $\mu$-accessible.
\end{proof}

\begin{rem}\label{estimate}
Let $\mu$ be a regular cardinal such that $\lambda\trianglelefteq\mu$ and $|\ca|<\mu$ where $|\ca|$ denotes the cardinality of the set
of morphisms of $\ca$. Following \cite[2.33, 2.34]{adamek-rosicky}, $\ck_{mono}$ is $(\mu^{<\mu})^+$-accessible.
\end{rem}

When the starting class is locally multipresentable, the index of accessibility is preserved:

\begin{lem}\label{mono1}
Let $\ck$ be a locally $\lambda$-multipresentable category. Then $\ck_{mono}$ is locally $\lambda$-multipresentable.
\end{lem}
\begin{proof}
Following \ref{mono}, we have to prove that each object $K$ of $\ck$ is a $\lambda$-directed colimit of $\lambda$-presentable objects 
in $\ck_{mono}$. Let $a_i:A_i\to K$ be a $\lambda$-directed colimit of $\lambda$-presentable objects $A_i$ in $\ck$. Following \cite[4.5]{kelly}, $\ck$ has (strong epimorphism, monomorphism)-factorizations; $\ck$ is well-powered by
\cite[A1.4.17]{johnstone}. Let $a_i=a_i''a_i'$ be a (strong epimorphism, monomorphism) factorization of $a_i$. Then $a_i'':B_i\to K$ is 
a $\lambda$-directed colimit of monomorphisms. Following \cite[1.69]{adamek-rosicky}, objects $B_i$ are $\lambda$-presentable in $\ck_{mono}$.
Thus $\ck_{mono}$ is $\lambda$-accessible.

It remains to show that $\ck_{mono}$ has connected limits. Since equalizers are monomorphisms, we only need that wide pullbacks 
of monomorphisms are monomorphisms. But this is evident.
\end{proof}

\begin{remark}\label{compare}
If $K$ is $\lambda$-presentable in $\ck$ then it is $\lambda$-presentable in $\ck_{mono}$. One cannot expect the converse: in a general locally $\lambda$-presentable category $\ck$, for example, the objects of $\ck$ that are $\lambda$-presentable in $\ck_{mono}$ (the {\it $\lambda$-generated objects} of $\ck$) are guaranteed only to be strong quotients of $\lambda$-presentables (see 
\cite[1.67, 1.68]{adamek-rosicky}).\end{remark}
 
\begin{cor}\label{evcatlocnullpres} The eventual categoricity conjecture in the sense of internal sizes (see \cite{multipres-pams}) holds for locally $\aleph_0$-multipresentable categories.\end{cor}
\begin{proof} By contrast with the remark above, whenever $\ck$ is locally $\aleph_0$-multipresentable, the embedding $\ck_{mono}\to\ck$ preserves $\mu$-presentable objects 
for all sufficiently large $\mu$ (see \cite[4.3]{beke-rosicky}). Recalling that any locally $\aleph_0$-multipresentable category is a universal AEC (\cite[5.9]{multipres-pams}), the eventual categoricity conjecture for locally $\aleph_0$-multipresentable categories follows from the corresponding result for universal AECs in \cite{ap-universal-apal, categ-universal-2-selecta}. 
\end{proof}

In order to be able to restrict to \emph{regular} monomorphisms, we will assume existence of pushouts:

\begin{lem}\label{regular}
Let $\ck$ be a $\lambda$-accessible category with pushouts. Then $\ck_{reg}$ is accessible and has $\lambda$-directed colimits.
\end{lem}
\begin{proof}
We follow the proof of \ref{mono}. We note, first, that regular monomorphisms are stable under $\lambda$-directed colimits 
in $\ck$ (see \cite{on-pure-morphisms} Proposition 2).

For the rest, it suffices to note that $\lambda$-pure morphisms are regular monomorphisms (see \cite[Corollary 1]{on-pure-morphisms}).
\end{proof}

\section{Some model theory of $\mu$-AECs}\label{mt-mu-aec}

We now shift more firmly to the model-theoretic context, establishing certain technical results on $\mu$-AECs that we will require for the proofs related to canonicity of stable independence relations, and their relationship to a suitable formulation of the order property.  As mentioned already, Lemma~\ref{mono} and Fact~\ref{mu-aec-acc} will allow us to translate such results back to the broader category-theoretic framework that we have considered up to this point. 

Roughly speaking, the following lemma says that under reasonable assumptions, presentability and cardinality very often coincide quite often in reasonable classes of structures. Parts of the proof are similar to (for example) \cite[2.3.10]{makkai-pare}. The lemma will be used in the proof of Theorem \ref{mu-aec-criteria} and the proof of Theorem \ref{accessible-charact}. Note that we do not assume that the class $\K$ is coherent, nor that $M$ is a $\tau (\K)$-substructure of $N$ when $M \lea N$.

\begin{lem}\label{mu-aec-technical-2}
    Let $\K = (K, \lea)$ be such that:
    \begin{enumerate}
    \item $K$ is a class of structures in a fixed $\mu$-ary vocabulary $\tau (\K)$.
    \item $\lea$ is a partial order on $K$ such that $M \lea N$ implies $U M \subseteq U N$.
    \item $\K$ is closed under isomorphisms.
    \end{enumerate}

    Let $\mu \le \theta$ be regular cardinals such that $\K$ is $\theta$-accessible and has concrete $\mu$-directed colimits. Let $C$ be the class of cardinals $\lambda$ such that for any $M \in \K$, $|U M| < \lambda$ if and only if $M$ is $\lambda_0$-presentable for some $\lambda_0 < \lambda$. Then:

  \begin{enumerate}
  \item $C$ is closed unbounded.
  \item\label{mu-aec-technical-2-2} If $\lambda$ is such that $\lambda = \lambda^{<\mu}$ and $\theta \le \lambda^+$, then for any $M \in \K$ there exists $\seq{M_i : i \in \mathcal{I}}$ increasing and $\lambda^+$-directed such that:
    \begin{enumerate}
    \item $M = \bigcup_{i \in \mathcal{I}} M_i$.
    \item For all $i \in \mathcal{I}$, $M_i$ is a $\mu$-directed union of at most $\lambda$-many $\theta$-presentable objects. In particular, $M_i$ is $\lambda^+$-presentable and if in addition $(\theta, \lambda^+] \cap C \neq \emptyset$, then also $|U M_i| \le \lambda$.
    \item $M_i = \bigcup_{j < i} M_j$ for all $i \in \mathcal{I}$ such that $\{j \in \mathcal{I} \mid j < i\}$ is $\mu$-directed.
    \end{enumerate}
  \item If $\lambda > \theta$ is such that $\lambda = \lambda^{<\mu}$, and $C \cap (\theta, \lambda] \neq \emptyset$, then $\lambda^+ \in C$.
  \end{enumerate}
\end{lem}
\begin{proof} \
  \begin{enumerate}
  \item $C$ is clearly closed. Now given a cardinal $\lambda_0$, build $\seq{\lambda_i : i \le \omega}$ increasing continuous such that for all $i < \omega$, for any $M \in \K$, if $M$ is $\lambda_i$-presentable, then $|U M| < \lambda_{i + 1}$ and if $|U M| \le \lambda_i$, then $M$ is $\lambda_{i + 1}$-presentable. This is possible since there is up to isomorphism only a set of objects of cardinality at most $\lambda_i$ and a set of $\lambda_i$-presentable objects. Now $\lambda_{\omega}$ is in $C$, as desired.
  \item Let $M \in \K$. We know that $M$ can be written as a $\theta$-directed union of $\theta$-presentable objects $M = \bigcup_{i \in I} N_i$. Now by the cardinal arithmetic assumption, any subset of $I$ of cardinality at most $\lambda$ can be completed to a $\mu$-directed subset of $I$ of cardinality at most $\lambda$. Thus letting $\mathcal{I}$ be the set of $\mu$-directed subsets of $I$ of cardinality at most $\lambda$, we have that $M = \bigcup_{J \in \mathcal{I}} N_J$. This is a $\lambda^+$-directed system, thus letting for $i \in \mathcal{I}$ $M_i := N_i$, the $M_i$'s are as desired. As for the ``in particular'' part, it is routine to check that the $M_i$'s are $\lambda^+$-presentable (see \cite[3.5(2)]{internal-sizes-v3}) and if $(\theta, \lambda^+] \cap C \neq \emptyset$ this means that $M_i$ is a union of at most $\lambda$-many objects of cardinality at most $\lambda$, so $|U M_i| \le \lambda$.
  \item Let $M \in \K$. Let $\seq{M_i : i \in \mathcal{I}}$ be as given by the previous part. We show that $|U M| \le \lambda$ if and only if $M$ is $\lambda^+$-presentable. First, if $|U M| \le \lambda$ then as $\mathcal{I}$ is $\lambda^+$-directed $M = M_i$ for some $i \in \mathcal{I}$ and since $M_i$ is $\lambda^+$-presentable we are done. Conversely if $M$ is $\lambda^+$-presentable then again as $\mathcal{I}$ is $\lambda^+$-directed we must have that $M = M_i$ for some $i \in \mathcal{I}$, hence $|U M| = |U M_i| \le \lambda$, as desired.
  \end{enumerate}
\end{proof}

The next result is similar to \cite[4.5]{mu-aec-jpaa}, but we do \emph{not} require anything on how presentability ranks relate to cardinalities. Rather, what we need is derived via Lemma \ref{mu-aec-technical-2}. This also improves \cite[5.5]{beke-rosicky}, which is the case $\mu = \aleph_0$.

\begin{thm}\label{mu-aec-criteria}
  Let $\K$ be a coherent abstract class in a $\mu$-ary vocabulary with concrete $\mu$-directed colimits. Then $\K$ is accessible if and only if $\K$ is a $\mu$-AEC.
\end{thm}
\begin{proof}
  If $\K$ is a $\mu$-AEC, then by Fact \ref{mu-aec-acc}, $\K$ is an accessible category. We show the converse. Assume that $\K$ is accessible and fix $\theta \ge \mu$ such that $\K$ is $\theta$-accessible (such a $\theta$ exists by \cite[2.3.10]{makkai-pare}). Let $C$ be as given by Lemma \ref{mu-aec-technical-2}. We have to show that $\LS (\K)$ exists. Let $\lambda_0 := \left(\min (C \cap [\theta^+, \infty))\right)^{<\mu}$. Note that $\theta < \lambda_0 = \lambda_0^{<\mu}$. We claim that $\LS (\K) \le \lambda_0$. Let $M \in \K$ and let $A \subseteq |M|$. Let $\lambda := |A|^{<\mu} + \lambda_0$. By Lemma \ref{mu-aec-technical-2}(\ref{mu-aec-technical-2-2}), there exists a $\lambda^+$-directed resolution $\seq{M_i : i \in \mathcal{I}}$ of $M$ such that each $M_i$ has cardinality at most $\lambda$. Since the resolution is $\lambda^+$-directed and $|A| \le \lambda$, there exists $i \in \mathcal{I}$ such that $A \subseteq U M_i$. Since $|U M_i| \le \lambda = \lambda_0 + |A|^{<\mu}$, we are done.
\end{proof}

\subsection{Model-homogeneous models}

The following notions are very close to ``$\kappa$-closed'' and ``$\kappa$-saturated'' from \cite{rosicky-sat-jsl}. The only difference is that we use cardinalities instead of presentability ranks. Most of the proofs there carry through with little change, however.

\begin{defin}\label{mh-def}
  Let $\K$ be a $\mu$-AEC. For $\kappa > \LS (\K)$, $M \lea N$, we say that $M$ is \emph{$\kappa$-model-homogeneous in $N$} if whenever $M_0, N_0 \in \K$ are such that:

  \begin{enumerate}
  \item $M_0 \lea M$.
  \item $M_0 \lea N_0 \lea N$.
  \item $N_0 \in \K_{<\kappa}$
  \end{enumerate}

  Then there exists $f: N_0 \rightarrow M$ fixing $M_0$.

  We say that $M$ is \emph{locally $\kappa$-model-homogeneous} if it is $\kappa$-model-homogeneous in $N$ for every $N \in \K$ with $M \lea N$. We say that $M$ is \emph{$\kappa$-model-homogeneous} if whenever $M_0 \lea M$, $M_0 \lea N_0$ are such that $N_0 \in \K_{<\kappa}$, there exists $f: N_0 \rightarrow M$ fixing $M_0$.
\end{defin}

\begin{remark}\label{mh-rmk}
  If $\K$ has amalgamation, being locally $\kappa$-model-homogeneous is equivalent to being $\kappa$-model-homogeneous (see \cite[Lemma 3]{rosicky-sat-jsl}). Further (still assuming amalgamation), if $M$ is $\kappa$-model-homogeneous in $N$ and $N$ is $\kappa$-model-homogeneous, then $M$ is $\kappa$-model-homogeneous. 
\end{remark}

\begin{defin}
  Let $\K$ be a $\mu$-AEC and let $\kappa > \LS (\K)$. We let $\Kmh$ be the abstract class of locally $\kappa$-model-homogeneous models in $\K$, ordered with the appropriate restriction of $\lea$.
\end{defin}

Given an increasing chain $\seq{M_i : i < \delta}$ in a $\mu$-AEC $\K$, the union of the chain may not be inside $\K$. The following conditions are a very useful weakening: we require that this chain have an upper bound. This is already used in \cite{rosicky-sat-jsl}. Classes of saturated models in AECs, as well as $\mu$-CAECs (e.g.\ metric classes) are examples satisfying this property, see \cite[6.7]{mu-aec-jpaa}.

\begin{defin}\label{directed-def}
  A $\mu$-AEC $\K$ has \emph{directed bounds} if whenever $\seq{M_i : i \in I}$ is a directed system, there exists $M \in \K$ such that $M_i \lea M$ for all $i \in I$.

  We say that $\K$ \emph{has chain bounds} if this holds whenever $I$ is an ordinal.
\end{defin}

In practice, we will use chain bounds. It allows us to build locally $\kappa$-model-homogeneous models. Moreover a large cardinal axiom ({\it Vop\v enka's Principle}, see Chapter 6 of \cite{adamek-rosicky}) ensures that the class of all such models is well-behaved. In particular, it has the amalgamation property.

\begin{thm}\label{mh-thm}
  Let $\K$ be a $\mu$-AEC and let $\kappa > \LS (\K)$ be regular.

  \begin{enumerate}
  \item Assume that $\K$ has chain bounds. For any $M \in \K$, there exists $N \in \K$ such that $M \lea N$ and $N$ is locally $\kappa$-model-homogeneous. Moreover, we can take $N$ so that $| U N| \le \left(|U M| + \LS (\K)\right)^{<\kappa}$.
  \item\label{mh-thm-2} For any $N \in \K$ and any $A \subseteq U N$, there exists $M \in \K$ with $M \lea N$ such that $M$ contains $A$, $M$ is $\kappa$-model-homogeneous in $N$, and $|U M| \le \left(|A| + \LS (\K)\right)^{<\kappa}$.
  \item Assume that $\K$ has chain bounds. If either $\K$ has amalgamation or Vop\v enka's principle holds, then $\Kmh$ is a $\kappa$-AEC with chain bounds.
  \item If $\kappa$ is strongly compact, then any locally $\kappa$-model-homogeneous model $M$ is a global amalgamation base: whenever $M \lea M_\ell$, $\ell = 1,2$, there exists $N \in \K$ and $f_\ell : M_\ell \rightarrow N$ fixing $M$.
  \end{enumerate}
\end{thm}
\begin{proof} \
  \begin{enumerate}
  \item Similar to \cite[Theorem 1]{rosicky-sat-jsl}.
  \item Similar to the above.
  \item All the axioms are straightforward, except for the Löwenheim-Skolem-Tarski axiom. If $\K$ has amalgamation, this follows from Theorem \ref{mh-thm} and Remark \ref{mh-rmk}. If Vop\v enka's principle holds, then by \cite[4.6]{mu-aec-jpaa}, the Löwenheim-Skolem-Tarski axiom holds (though we do not have a bound for it).
  \item As in \cite[7.2]{tamelc-jsl}, using the fact that one can take ultraproducts in $\mu$-AECs \cite[\S5]{mu-aec-jpaa}
  \end{enumerate}
\end{proof}

Now that we have tools to get amalgamation, we investigate to what extent we can work inside a big homogeneous model as in the first-order case (and, indeed, as in AECs with amalgamation):

\begin{defin}
  For $\K$ a $\mu$-AEC and $\kappa > \LS (\K)$, we say that $M \in \K$ is \emph{$\kappa$-universal} if any $M_0 \in \K_{<\kappa}$ embeds into $M$.
\end{defin}

\begin{defin}\label{monster-def}
  We say that a $\mu$-AEC $\K$ \emph{has monster models} if for any $\kappa > \LS (\K)$ there exists $M \in \K$ which is both $\kappa$-universal and $\kappa$-model-homogeneous.
\end{defin}

It turns out that building monster models requires (at least when $\mu = \aleph_1$) chain bounds:

\begin{thm}\label{monster-thm}
  Let $\K$ be a $\mu$-AEC. If $\K$ is non-empty, has amalgamation, joint embedding, and chain bounds, then $\K$ has monster models. Conversely, if $\K$ has monster models, then $\K$ has amalgamation and joint embedding; if $\mu \le \aleph_1$, moreover, $\K$ has chain bounds.
\end{thm}
\begin{proof}
  If $\K$ has amalgamation, joint embedding, and has chain bounds, then the construction of a $\kappa$-universal $\kappa$-model-homogeneous model is standard.

  Conversely, if $\K$ has monster models, then it clearly has amalgamation and joint embedding. To see it has chain bounds when $\mu = \omega_1$, fix $\seq{M_i : i < \delta}$ increasing. Without loss of generality, $\delta = \cf{\delta}$. If $\delta \ge \omega_1$, we can use the chain axioms of $\aleph_1$-AECs, so assume without loss of generality that $\delta = \omega$. Let $M$ be $\kappa$-universal and $\kappa$-model-homogeneous, where $\kappa := \left(\sum_{i < \delta} |U M_i|^{<\mu} + \LS (\K)\right)^+$. We build $\seq{f_i: M_i \rightarrow M \mid i < \omega}$ increasing as follows:

  \begin{enumerate}
  \item For $i = 0$, use universality of $M$.
  \item For $i$ successor, use model-homogeneity of $M$.
  \end{enumerate}

  Now let $N \lea M$ contain $\bigcup_{i < \omega} f_i[M_i]$ and rename to obtain an upper bound to $\seq{M_i : i < \omega}$.
\end{proof}

Assuming a large cardinal axiom, then, there is a well-behaved sub-$\mu$-AEC of the original class:

\begin{cor}\label{kast-cor}
  Let $\K$ be a $\mu$-AEC, let $\kappa > \LS (\K)$ be strongly compact and assume Vop\v enka's principle. If $\K$ has chain bounds, then there exists a subclass $\K^\ast$ of $\K$ such that:

  \begin{enumerate}
  \item $\K^\ast$ is a $\kappa$-AEC.
  \item $\K^\ast$ has amalgamation, joint embedding, and chain bounds.
  \item If $\K$ is not empty, $\K^\ast$ is not empty and can be chosen to have arbitrarily large models if $\K$ has arbitrarily large models. In this case, $\K^\ast$ will have monster models.
  \item If $\K$ has joint embedding, then $\K^\ast$ is cofinal in $\K$, i.e.\ any $M \in \K$ is contained inside an $N \in \K^\ast$.
  \end{enumerate}
\end{cor}
\begin{proof}
  By Theorem \ref{mh-thm}, $\Kmh$ is a $\kappa$-AEC with amalgamation and chain bounds. If $\K$ had joint embedding already, then $\Kmh$ also had joint embedding, and hence one can take $\K^\ast := \Kmh$. Otherwise, using the equivalence relation induced by ``embedding in a common model,'' one can partition $\Kmh$ into disjoint $\kappa$-AECs $\seq{\K_i^\ast : i \in I}$ such that each has amalgamation, joint embedding, and chain bounds. If $\K$ has arbitrarily large models, $\Kmh$ has arbitrarily large models and hence (since $I$ is a set) we can pick $i \in I$ such that $\K_i^\ast$ has arbitrarily large models. Let $\K^\ast = \K_i^\ast$.
\end{proof}

\begin{question}
  Can one remove Vop\v enka's principle from Corollary \ref{kast-cor}?
\end{question}

One approach would be to consider a stronger ordering on $\K$, e.g.\ $M \tleq N$ if and only if $M \lea N$ and $M \lee_{\Ll_{\kappa, \kappa}} N$. In this case, however, we do not know whether having chain bounds is preserved.

\section{Stable amalgamation inside a $\mu$-AEC}\label{stable-mu-aec}

In this section, we consider independence relations on $\mu$-AECs. The main result is Theorem \ref{accessible-charact}, characterizing stable independence in terms of model-theoretic local character properties of forking. All throughout, we assume:

\begin{hypothesis}
  We work inside a fixed $\mu$-AEC $\K$. We fix an invariant (Definition \ref{invariant-def}) independence relation $\nf$ on $\K$.
\end{hypothesis}

We can see $\nf$ as a relation on Galois types if we introduce some notation. A similar idea is already investigated in \cite[\S5.1]{bgkv-apal}, but there the left hand side of $\nf$ is already assumed to be an arbitrary set. Thus the situation here requires slightly more caution.

\begin{defin}
  Write $\nfcl{N_0}{A}{B}{N_3}$ if $N_0 \lea N_3$, $A \cup B \subseteq U N_3$, and there exists $M_1, M_2, M_3$ with $A \subseteq U M_1$, $B \subseteq U M_2$, $N_3 \lea M_3$, and $\nfs{N_0}{M_1}{M_2}{M_3}$.

  We say that $\gtp (\ba / B; N_3)$ \emph{does not fork over $N_0$} if $\nfcl{N_0}{\ran{\ba}}{B}{N_3}$ (it is easy to see that this does not depend on the choice of representatives, see also Fact \ref{nfcl-facts}(\ref{nfcl-k-embed})).
\end{defin}

One can think of $\nfm$ as the ``closure'' of $\nf$. The point is that we allow sets on the left and right hand side.

\begin{remark}
  It is tempting to try to define $\nfm$ in an arbitrary category as the class of diagrams

  $$  
  \xymatrix@=3pc{
    M_1 \ar[r] & M_3 \\
    M_0 \ar [ur] &
    M_2 \ar[u]
  }
  $$

  that can be extended to an independent diagram consisting of $M_0, M_1', M_2'$ and $M_3'$ such that the following commutes:

  $$  
  \xymatrix@=3pc{
    & M_1' \ar@{.>}[r] & M_3' \\
    M_1 \ar[r] \ar@{.>}[ru] & M_3 \ar@{.>}[ru] & M_2' \ar@{.>}[u]\\
    M_0 \ar@{.>}[uur] \ar@{.>}[urr] \ar [ur] & M_2 \ar[u] \ar@{.>}[ru]
  }
  $$

  However, when there is already a morphism e.g.\ from $M_0$ to $M_2$, there is no reason to believe that the resulting diagram will also commute with that morphism. This is problematic when one tries to prove, for example, that $\nfm$ is transitive (see Theorem \ref{nfcl-uq}). Thus it is unclear to us how to define $\nfm$ in a non-concrete category.
\end{remark}

Properties of $\nf$ generalize to $\nfm$ as follows:

\begin{fact}\label{nfcl-facts}
  Assume that $\nf$ is monotonic.
  
  \begin{enumerate}
  \item Let $M_0 \lea M_\ell \lea M_3$ for $\ell = 1, 2$. Then $\nfs{M_0}{M_1}{M_2}{M_3}$ if and only if $\nfcl{M_0}{M_1}{M_2}{M_3}$.
  \item\label{nfcl-k-embed} (Preservation under $\K$-embeddings) Given $M_0 \lea M_3$, $A, B \subseteq U M_3$, and $f: M_3 \rightarrow N_3$, we have that $\nfcl{M_0}{A}{B}{M_3}$ if and only if $\nfcl{f[M_0]}{f[A]}{f[B]}{N_3}$.
  \item (Monotonicity) If $\nfcl{M_0}{A}{B}{M_3}$ and $A_0 \subseteq A$, $B_0 \subseteq B$, then $\nfcl{M_0}{A_0}{B_0}{M_3}$.
  \item (Normality) $\nfcl{M_0}{A}{B}{M_3}$ if and only if $\nfcl{M_0}{A M_0}{B M_0}{M_3}$.
  \item (Base monotonicity) Assume that $\nf$ is right base-monotonic. If $\nfcl{M_0}{A}{B}{M_3}$, $M_0 \lea M_2 \lea M_3$, and $U M_2 \subseteq B$, then $\nfcl{M_2}{A}{B}{M_3}$.
  \item (Extension) Assume that $\nf$ has existence. Whenever $M \lea N$ and $p \in \gS^{<\infty} (M)$, there exists $q \in \gS^{<\infty} (N)$ extending $p$ such that $q$ does not fork over $M$.
  \item (Symmetry) Assume that $\nf$ is symmetric. Then $\nfm$ is symmetric: $\nfcl{M}{A}{B}{N}$ holds if and only if $\nfcl{M}{B}{A}{N}$ holds.
  \end{enumerate}
\end{fact}
\begin{proof}
  This is essentially given by the arguments in \cite[5.1,5.4]{bgkv-apal}. For the convenience of the reader, we sketch some details:

  \begin{enumerate}
  \item If $\nfs{M_0}{M_1}{M_2}{M_3}$, then directly from the definition $\nfcl{M_0}{M_1}{M_2}{M_3}$. For the converse, use monotonicity, invariance, closure under $\sim$, the coherence axiom of $\mu$-AECs, and the definition of $\nfm$.
  \item Directly from the definitions, invariance, and closure under $\sim$.
  \item Clear from the definition of $\nfm$.
  \item The right to left direction is by monotonicity. For the left to right direction, suppose that we have $M_1, M_2, M_3'$ witnessing $\nfcl{M_0}{A}{B}{M_3}$, i.e.\ $M_0 \lea M_\ell \lea M_3'$ for $\ell = 1,2$, $M_3 \lea M_3'$, $A \subseteq U M_1$, $B \subseteq U M_2$, and $\nfs{M_0}{M_1}{M_2}{M_3'}$. By the first part, $\nfcl{M_0}{M_1}{M_2}{M_3'}$. By monotonicity, $\nfcl{M_0}{A M_0}{BM_0}{M_3'}$. Since $\nfm$ preserves $\K$-embeddings, this implies that $\nfcl{M_0}{AM_0}{BM_0}{M_3}$, as desired.
  \item Directly from the definition and the coherence axiom of $\mu$-AECs.
  \item Say $p = \gtp (\bb / M; M')$. By existence, find $N' \in \K$ and $f: M' \rightarrow N'$ such that $N \lea N'$, $f$ fixes $M$, and $\nfs{M}{f[M']}{N}{N'}$. Let $q := \gtp (f (\bb) / N; N')$.
  \item Directly from the definition.
    \end{enumerate}
\end{proof}

Uniqueness and transitivity also generalize, but the argument is more involved than the corresponding one in \cite[5.4, 5.11]{bgkv-apal} (essentially because we are taking a closure with respect to both the left and right hand side of $\nf$). We sketch a full proof here. 

\begin{thm}\label{nfcl-uq}
  Assume that $\nf$ is transitive and has existence and uniqueness.
  
  \begin{enumerate}
  \item (Uniqueness) Given $p, q \in \gS^{<\infty} (B; N)$ with $M \lea N$ and $U M \subseteq B \subseteq U N$, if $p \rest M = q \rest M$ and $p$, $q$ do not fork over $M$, then $p = q$.
  \item (Transitivity) If $M_0 \lea M_2 \lea M_3$, $\nfcl{M_0}{A}{M_2}{M_3}$ and $\nfcl{M_2}{A}{B}{M_3}$, then $\nfcl{M_0}{A}{B}{M_3}$.
  \end{enumerate}
\end{thm}
\begin{proof}
  We proceed via a series of claims.

  \underline{Claim 1}: Assume that $\nfs{M_0}{M_1}{M_2}{M_3}$ and $\nfs{M_0}{M_1'}{M_2}{M_3'}$. Let $\bb, \bb'$ be enumerations of $M_1, M_1'$ respectively. If $\gtp (\bb / M_0; M_3) = \gtp (\bb' / M_0; M_3')$ (i.e.\ the enumerations induce an isomorphism between $M_1$ and $M_1'$ fixing $M_0$), then $\gtp (\bb / M_2; M_3) = \gtp (\bb' / M_2; M_3')$.

  \underline{Proof of Claim 1}: This follows directly from the uniqueness property as in \cite[12.6]{indep-aec-apal}. $\dagger_{\text{Claim 1}}$.

  \underline{Claim 2}: Assume that $\nfs{M_0}{M_2}{M_1}{M_3}$ and $\nfs{M_0}{M_2}{M_1'}{M_3}$. There exists $M_1''$, $M_3'$, $f: M_1' \rightarrow M_1''$ such that $f$ fixes $M_0$, $M_1 \lea M_1'' \lea M_3'$, $M_3 \lea M_3'$, and $\nfs{M_0}{M_2}{M_1''}{M_3'}$.

  \underline{Proof of Claim 2}: Let $\ba$ be an enumeration of $M_2$. Let $p := \gtp (\ba / M_0; M_3)$. By extension, let $q \in \gS^{<\infty} (M_3)$ be such that $q$ extends $p$ and $q$ does not fork over $M_0$. Say $q = \gtp (\ba' / M_3; M_3'^\ast)$, and let $M_2'$ be the model enumerated by $\ba'$. We have that $\nfs{M_0}{M_2'}{M_3}{M_3^\ast}$. By monotonicity, $\nfs{M_0}{M_2'}{M_1}{M_3^\ast}$. By Claim 1 (where the role of $M_1$ and $M_2$ is reversed), $\gtp (\ba / M_1; M_3) = \gtp (\ba' / M_1; M_3^\ast)$. Let $g: M_3^\ast \rightarrow M_3'$ with $M_3 \lea M_3'$ be such that $g (\ba') = \ba$ and $g$ fixes $M_1$. Let $M_1'' := g[M_3]$ and let $f := g \rest M_1'$. This is as desired. $\dagger_{\text{Claim 2}}$.

  \underline{Claim 3}: Assume that $\nfs{M_0}{M_1}{M_2}{M_3}$ and $\bb_1, \bb_2 \in \fct{<\infty}{M_1}$ are such that $\gtp (\bb_1 / M_0; M_3) = \gtp (\bb_2 / M_0; M_3)$. Then $\gtp (\bb_1 / M_2; M_3) = \gtp (\bb_2 / M_2; M_3)$.

  \underline{Proof of Claim 3}: Let $f: M_3 \rightarrow M_3'$, $M_3 \lea M_3'$ be such that $f$ fixes $M_0$ and $f (\bb_1) = \bb_2$. Using extension, find $g: M_3' \rightarrow M_3''$ that fixes $M_1$, such that $M_3 \lea M_3''$ and $\nfs{M_1}{g[M_3']}{M_3}{M_3''}$. By transitivity (for $\nf$), $\nfs{M_0}{g[M_3']}{M_2}{M_3}$. Let $h := gf$. We have in particular that $\nfs{M_0}{h[M_1]}{M_2}{M_3}$. Letting $\ba$ be an enumeration of $M_1$ and using Claim 1, this means that $\gtp (\ba / M_2; M_3) = \gtp (h (\ba) / M_2; M_3)$. In particular, $\gtp (\bb_1 / M_2; M_3) = \gtp (\bb_2 / M_2; M_3)$. $\dagger_{\text{Claim 3}}$.

  \underline{Claim 4 (uniqueness for types over models)}: Let $M_0 \lea M_2$ and let $p_1, p_2 \in \gS^{<\infty} (M_2)$ be given such that both do not fork over $M_0$ and $p_1 \rest M_0 = p_2 \rest M_0$. Then $p_1 = p_2$.

  \underline{Proof of Claim 4}: Say $p_\ell = \gtp (\bb_\ell / M_2; N_\ell)$. Without loss of generality $N := N_1 = N_2$). By definition of nonforking of types, $\nfcl{M_0}{\bb_\ell}{M_2}{N}$. Expanding the definition of $\nfm$ and extending $N$ if necessary, we have that for some $M_1^\ell$ containing $\bb_\ell$, $\ell = 1,2$, $\nfs{M_0}{M_1^\ell}{M_2}{N}$. By Claim 2 applied to $\nf^d$, there exists $M_1'$, $N'$, and $f: M_1^2 \rightarrow M_1'$ such that $f$ fixes $M_0$, $M_1^1 \lea M_1'$, $N \lea N'$, and $\nfs{M_0}{M_1'}{M_2}{N'}$. Let $\bb_2' := f (\bb_2)$. By Claim 3, $\gtp (\bb_1 / M_2; N') = \gtp (\bb_2' / M_2; N')$. Moreover, Claim 1 implies that $\gtp (\bb_2' / M_2; N') = \gtp (\bb_2 / M_2; N')$, so we get the desired result. $\dagger_{\text{Claim 4}}$

  \underline{Claim 5 (transitivity)}: Let $M_0 \lea M_1 \lea M_2$ and let $p \in \gS^{<\infty} (M_2)$. If $p$ does not fork over $M_1$ and $p \rest M_1$ does not fork over $M_0$, then $p$ does not fork over $M_0$ (note that this implies the transitivity statement of the theorem, since by definition of $\nfm$ we can always extend $B$ so that it is a model extending $M_2$).

  \underline{Proof of Claim 5}: By extension, let $q \in \gS^{<\infty} (M_2)$ extend $p \rest M_0$ such that $q$ does not fork over $M_0$. By monotonicity, $q \rest M_1$ does not fork over $M_0$. By Claim 4 (applied to $q \rest M_1$ and $p \rest M_1)$, $p \rest M_1 = q \rest M_1$. By base monotonicity, $q$ does not fork over $M_1$. By Claim 4 again, $p = q$. $\dagger_{\text{Claim 5}}$

  \underline{Claim 6}: Let $p \in \gS^{<\infty} (B; N)$ with $M \lea N$ and $U M \subseteq B \subseteq U N$. If $p$ does not fork over $M$, then there exists $q \in \gS^{<\infty} (N)$ such that $q$ extends $p$ and $q$ does not fork over $M$.

  \underline{Proof of Claim 6}: First note that this is not quite the same as the extension property, since we start of with a type that already does not fork over a smaller base. Write $p = \gtp (\ba / B; N)$. Fix $M', N'$ such that $M \lea M' \lea N'$, $N \lea N'$, $B \subseteq U M'$, and $\nfcl{M}{\ba}{M'}{N'}$. Let $p' := \gtp (\ba / M'; N')$. Note that $p'$ extends $p$ and $p'$ does not fork over $M$. Find $q' \in \gS^{<\infty} (N')$ such that $q'$ extends $p'$ and $q'$ does not fork over $M'$. By Claim 5 (transitivity), $q'$ does not fork over $M$. Let $q := q' \rest N$. By monotonicity, it is as desired. $\dagger_{\text{Claim 7}}$
  
  \underline{Claim 7 (uniqueness)}:  Given $p_1, p_2 \in \gS^{<\infty} (B; N)$ with $M \lea N$ and $U M \subseteq B \subseteq U N$, if $p_1 \rest M = p_2 \rest M$ and $p_1$, $p_2$ do not fork over $M$, then $p_1 = p_2$.

  \underline{Proof of Claim 7}: Using Claim 6, find $q_\ell \in \gS^{<\infty} (N)$ such that $q_\ell$ does not fork over $M$ and extends $p_\ell$ for $\ell = 1,2$. Thus in particular $q_1 \rest M = p_1 \rest M = p_2 \rest M = q_2 \rest M$. By Claim 4 (uniqueness for types over models), $q_1 = q_2$. In particular, $p_1 = q_1 \rest B = q_2 \rest B = p_2$. $\dagger_{\text{Claim 7}}$
\end{proof}

We can now state the local character property of forking: every type does not fork over a ``small'' set:

\begin{defin}\label{lc-def}
  We say that $\nf$ has \emph{right local character} if for each cardinal $\alpha$, there exists a cardinal $\lambda$ (depending on $\alpha$) such that for any $p \in \gS^{\alpha} (M)$ there exists $M_0 \lea M$ with $|U M_0| \le \lambda$ and $p$ not forking over $M_0$. Equivalently---avoiding any mention of Galois types---given $M \lea N$ and $N_1 \lea N$, if $N_1$ has cardinality at most $\alpha$ there exists $M_0$ of cardinality at most $\lambda$ and $N_1', N'$, such that $N_1 \lea N_1' \lea N'$, $N \lea N'$, and $\nfs{M_0}{N_1'}{M}{N'}$.

  We say that $\nf$ has \emph{left local character} if $\nf^d$ has right local character. We say that $\nf$ has \emph{local character} if it has both left and right local character.
\end{defin}

The idea of the definition of local character is as follows: given $M \lea N$ and $N_1 \lea N$, we may want to write that $\nfs{N_0}{N_1}{M}{N}$, where $N_0$ is the pullback of $N_1$ and $M$ over $N$. However pullbacks may not exist (and even when they do, the desired independence may not hold, see Example \ref{indep-examples}(\ref{acf-example})). Thus we say instead that we can close the intersection of $N_1$ and $M$ to a small (meaning of cardinality depending only on $\alpha$, the size of $N_1$) model $M_0$ so that $N_1$ and $M$ do not interact over $M_0$. Importantly, we still require that $M_0 \lea M$.

We will also study the following locality property, introduced as the ``model-witness property'' in \cite[3.12(9)]{indep-aec-apal}. When $\theta = \aleph_0$ (as in the first-order case), it is often called---somewhat confusingly---the finite character property of forking.

\begin{defin}\label{wp-def}
  Let $\theta$ be an infinite cardinal. We say that $\nf$ has the \emph{right $(<\theta)$-witness property} if $\nfs{M_0}{M_1}{M_2}{M_3}$ holds whenever $M_0 \lea M_\ell \lea M_3$, $\ell = 1,2$, and $\nfcl{M_0}{M_1}{A}{M_3}$ for all $A \subseteq U M_2$ with $|A| < \theta$.

  We say that $\nf$ has the \emph{left $(<\theta)$-witness property} if $\nf^d$ has it and we say that $\nf$ has the \emph{$(<\theta)$-witness property} if it has both the left and right one. When $\theta$ is omitted, we mean that the witness property holds for some $\theta$.
\end{defin}

The witness property is known to follow from appropriate tameness assumptions. Recall from \cite[2.8]{bg-apal} that $\K$ is \emph{fully $(<\theta)$-tame and short} if Galois types are determined by the restrictions of their domain and variables of size less than $\theta$. This holds in particular if $\theta \ge \kappa$, for some $\kappa > \LS (\K)$ strongly compact \cite[5.5]{mu-aec-jpaa} or if $\theta \ge \mu$ and $\K$ is a universal $\mu$-AEC (see the argument inside Remark \ref{nfastwitness}). For the convenience of the reader, we replicate the precise statement and proof of the witness property from full tameness and shortness here:

\begin{fact}[4.5 in \cite{indep-aec-apal}]\label{wp-fact}
  Let $\theta$ be an infinite cardinal and assume that $\nf$ is transitive and has existence and uniqueness. If $\K$ is fully $(<\theta$)-tame and short, then $\nf$ has the $(<\theta)$-witness property.
\end{fact}
\begin{proof}
  We prove the right $(<\theta)$-witness property. The left version follows by applying the same argument to $\nf^d$. Assume that $M_0 \lea M_\ell \lea M_3$, $\ell = 1,2$, and $\nfcl{M_0}{M_1}{A}{M_3}$ for all $A \subseteq U M_2$ with $|A| < \theta$. Let $\bb$ be an enumeration of $M_1$ and let $p := \gtp (\bb / M_2;M_3)$. By extension for $\nfm$, let $q \in \gS^{<\infty} (M_2)$ extend $p \rest M_0$ and not fork over $M_0$. We show that $p = q$, which is enough since $\nfm$ respects $\K$-embeddings. By full tameness and shortness, it is enough to see that $p \rest A = q \rest A$ for any $A$ of size less than $\theta$. Fix such an $A$. We know that $p \rest A$ does not fork over $M_0$ by assumption, and $q \rest A$ also does not fork over $M_0$ by monotonicity. Therefore by uniqueness for $\nfm$ (Theorem \ref{nfcl-uq}), $p \rest A = q \rest A$, as desired.
\end{proof}
\begin{remark}
  The argument does not need the full strength of tameness and shortness (in fact it only uses tameness, albeit for types of arbitrary length).
\end{remark}

\begin{example}\label{nfastwitness}
	Recall that for $\ck$ a coregular locally presentable category, we can define an independence relation $\nf^{\ast}$ on $\ck_{reg}$ consisting of effective pullback squares (see Definition~\ref{nf1-def}).  We note that $\nf^{\ast}$ has the witness property. To see this, use the fact that $\ck_{reg}$ is locally multipresentable, and that any locally multipresentable category whose morphisms are monomorphisms is fully tame and short (by the equivalence from \cite[5.9]{multipres-pams} and the proof of Boney's theorem that universal classes are tame, see \cite[3.7]{ap-universal-apal}), and hence Fact \ref{wp-fact} applies.
\end{example}

The witness property also makes the local character property more uniform: the cardinal $\lambda$ from Definition \ref{lc-def} becomes a simple function of $\alpha$:

\begin{lem}\label{wp-lc}
  Assume that $\nf$ is right transitive, and satisfies existence, uniqueness, and the left $(<\theta)$-witness property. If $\nf$ has right local character, then there exists a cardinal $\lambda_0$ such that for each $p \in \gS^\alpha (M)$, there exists $M_0 \in \K$ with $M_0 \lea M$, $|U M_0| \le \lambda_0 + \alpha^{<(\theta + \mu)}$ and $p$ not forking over $M_0$.
\end{lem}
\begin{proof}
  Let $\lambda_0$ be such that $\lambda_0 = \lambda_0^{<\mu} + \LS (\K)$ and satisfy Definition \ref{lc-def} with $\alpha, \lambda$ there standing for $\theta, \lambda_0$ here. Now given any $p \in \gS^{\alpha} (M)$, for any $I \subseteq \alpha$ with $|I| < \theta$, pick $M_I \lea M$ of cardinality at most $\lambda_0$ such that $p^I$ does not fork over $M_I$. In the end, let $A := \bigcup_{I \subseteq \alpha, |I| < \theta} M_I$. Note that $|A| \le \lambda_0 + \alpha^{<\theta}$, so one can pick $M_0 \lea M$ containing $A$ of size at most $|A|^{<\mu} + \LS (\K)$. By base monotonicity and the left $(<\theta)$-witness property, $M_0$ is as desired. 
\end{proof}

We now turn to the main result of this section. If $\nf$ is reasonable, accessibility of $\nf$ is equivalent to the conjunction of the witness and local character properties. We will use the following auxiliary class:

\begin{defin}
  Assume that $\nf$ is monotonic and right transitive. Let $\Knf = (\Knfnobf, \leanf)$ be the obvious coding of $\cknf$ into an abstract class: the vocabulary is $\tau (\Knf) = \tau (\K) \cup \{P\}$, where $P$ is a unary predicate, and we think of the members of $\Knfnobf$ as pairs of $\tau (\K)$-structures $(M, N)$ satisfying $M \lea N$ (so the elements of $M$ are the ones satisfying the predicate). We order $\Knf$ by $(M_0, M_1) \leanf (M_2, M_3)$ if and only if $\nfs{M_0}{M_1}{M_2}{M_3}$. 
\end{defin}
\begin{remark}\label{coherence-rmk}
  Assume that $\nf$ is monotonic and right transitive. Then $\Knf$ is isomorphic (as a category) to $\cknf$, so we need not distinguish between the two. We have that $\Knf$ is closed under isomorphisms, is a partial order, and $\Knf$ satisfies the coherence axiom: if $(M_0, M_1) \subseteq (M_0', M_1') \leanf (M, N)$ and $(M_0, M_1) \leanf (M, N)$, then by monotonicity (Lemma \ref{descent}) $(M_0, M_1) \leanf (M_0', M_1')$. However in $\Knf$, $\leanf$ may not refine the $\tau (\Knf)$-substructure relation: we could have $(M_0, M_1) \leanf (M, N)$ but $U M_0 \subsetneq U M \cap U M_1$.
\end{remark}

\begin{thm}[Characterization of stable independence]\label{accessible-charact}
  Let $\K$ be a $\mu$-AEC and let $\nf$ be an invariant independence relation on $\K$. Assume that $\nf$ is transitive, and has existence and uniqueness. The following are equivalent:

  \begin{enumerate}
  \item\label{accessible-charact-1} $\nf$ is accessible.
  \item\label{accessible-charact-15} For some $\lambda$, $\Knf$ satisfies all the axioms in the definition of a $\lambda$-AEC, except that $\leanf$ may not refine the $\tau (\Knf)$-substructure relation (see Remark \ref{coherence-rmk}).
  \item\label{accessible-charact-2} $\nf$ has the witness and local character properties.
  \end{enumerate}
\end{thm}

\begin{proof} \
  \begin{itemize}
  \item \underline{(\ref{accessible-charact-1}) implies (\ref{accessible-charact-15})}: Assume that $\nf$ is accessible. Pick $\lambda_0$ such that $\Knf$ is $\lambda_0$-accessible. By Remark \ref{acc2}, $\lambda_0$-directed colimits in $\K^2$ and $\Knf$ coincide, so in particular $\Knf$ has \emph{concrete} $\lambda_0$-directed colimits. By the proof of Theorem \ref{mu-aec-criteria} (here we are using that Lemma \ref{mu-aec-technical-2} does not require $\leanf$ to extend substructure), we obtain the desired result.
    \item \underline{(\ref{accessible-charact-15}) implies (\ref{accessible-charact-2})}: Fix a regular cardinal $\lambda$ such that $\Knf$ is a $\lambda$-AEC, except that $\leanf$ may not extend substructure. Fix $M_0, M_1, M_2, M_3$ such that $M_0 \lea M_\ell \lea M_3$, $\ell = 1,2$ and for any $A \subseteq U M_1$ with $|A| \le \LS (\Knf)$, $\nfcl{M_0}{A}{M_2}{M_3}$ holds. We show that $\nfs{M_0}{M_1}{M_2}{M_3}$, which will establish the $(<\LS (\Knf)^+)$-witness property. Since $\Knf$ is (almost) a $\lambda$-AEC, we can find a $\LS (\Knf)^+$-directed system $\seq{(M_0^i, M_1^i) : i \in I}$ such that $(M_0, M_1) = \bigcup_{i \in I} (M_0^i, M_1^i)$ and $|U M_1^i| \le \LS (\Knf)$ for all $i \in I$. This implies in particular that for any $i \in I$, $\nfs{M_0^i}{M_1^i}{M_0}{M_1}$. We also know that $\nfcl{M_0}{M_1^i}{M_2}{M_3}$. By transitivity for $\nfm$ (Theorem \ref{nfcl-uq}), $\nfs{M_0^i}{M_1^i}{M_2}{M_3}$. In other words, $(M_0^i, M_1^i) \leanf (M_2, M_3)$ for all $i \in I$. By smoothness, this implies that $(M_0, M_1) \leanf (M_2, M_3)$, i.e.\ $\nfs{M_0}{M_1}{M_2}{M_3}$, as desired.

Similarly, local character holds: let $M_0 \in \K$ and let $p \in \gS^{\alpha} (M_0)$. Without loss of generality $\alpha \le \LS (\K)$ (if not use the witness property). Say $p = \gtp (\ba / M_0; M_1)$. As before, find a $\LS (\Knf)^+$-directed system $\seq{(M_0^i, M_1^i) : i \in I}$ such that $(M_0, M_1) = \bigcup_{i \in I} (M_0^i, M_1^i)$ and $|UM_1^i| \le \LS (\Knf)$ for all $i \in I$. Since $I$ is $\LS (\Knf)^+$-directed, there is $i \in I$ such that $\ba \in \fct{\alpha}{M_1^i}$. Thus in particular $p$ does not fork over $M_0^i$. Since $|U M_0^i| \le \LS (\Knf)$, we are done.
  \item \underline{(\ref{accessible-charact-2}) implies (\ref{accessible-charact-1})}: Fix a regular cardinal $\theta$ such that $\nf$ has the $(<\theta)$-witness property and let $\lambda$ be as given by Lemma \ref{wp-lc}. We may assume without loss of generality that $\lambda = \lambda^{\theta + \mu + \LS (\K)}$ and $\lambda$ is regular. By Remark \ref{coherence-rmk}, it is enough to prove that $\Knf$ satisfies the Löwenheim-Skolem-Tarski and chain axioms of $\lambda$-AECs.

    \begin{itemize}
    \item \underline{LST axiom}: Let $(M, N) \in \Knf$ and let $A \subseteq U N$. Without loss of generality $|A|$ is infinite. We build $\seq{M_i, N_i : i < \lambda}$ increasing in $\K$ such that for all $i < \lambda$:
      \begin{enumerate}
      \item $A \subseteq U N_0$.
      \item $M_i \lea M$, $M_i \lea N_i \lea N$.
      \item $|U M_i| \le |U N_i| \le |A|^{<\lambda}$.
      \item $\nfcl{M_{i}}{\bigcup_{j < i} U N_j}{M}{N}$.
      \end{enumerate}

      This is enough: Let $M_\lambda := \bigcup_{i < \lambda} M_i$, $N_{\lambda} := \bigcup_{i < \lambda} N_i$. We claim that $\nfcl{M_\lambda}{N_\lambda}{M}{N}$. Indeed, recall that $\cf{\lambda} \ge \theta$ so for any $B \subseteq U N_\lambda$ with $|B| < \theta$, there exists $i < \lambda$ such that $B \subseteq U N_i$. In particular by monotonicity $\nfcl{M_{i + 1}}{B}{M}{N}$. By base monotonicity, $\nfcl{M_{\lambda}}{B}{M}{N}$. Since $B$ was arbitrary, by the $(<\theta)$-witness property we indeed have that $\nfcl{M_{\lambda}}{N_{\lambda}}{M}{N}$. However we also have that $M_{\lambda} \lea N_{\lambda}$, so $\nfs{M_{\lambda}}{N_{\lambda}}{M}{N}$, and hence $(M_\lambda, N_\lambda) \leanf (M, N)$, as needed.

      This is possible: given $\seq{M_j : j < i}$, $\seq{N_j : j < i}$, take any $M_i \lea M$ such that $|U M_i| \le |A|^{<\lambda}$, $M_j \lea M_i$ for all $j < i$, $A \cap U M \subseteq U M_i$, and $\nfcl{M_i}{\bigcup_{j < i} N_j}{M}{N}$ (use local character, Lemma \ref{wp-lc}, and base monotonicity). Now given $\seq{M_j : j \le i}$ and $\seq{N_j : j < i}$, pick any $N_i \lea N$ such that $N_j \lea N_i$ for all $j < i$, $|U N_i| \le |A|^{<\lambda}$, $M_{i} \lea N_i$, and $A \subseteq U N_i$.

    \item \underline{Chain axioms}: Fix $\seq{(M_i, N_i) : i \in I}$ a $\lambda$-directed system in $\Knf$ and let $(M, N) := \bigcup_{i \in I} (M_i, N_i)$. Clearly, $(M, N) \in \Knf$. We first want to see that $\nfs{M_i}{N_i}{M}{N}$. We use the witness property. Fix $A \subseteq U M$ of size less than $\theta$. Since $\theta \le \lambda$, there exists $j \in I$ such that $i \le j$ and $A \subseteq U M_j$. Since $\nfs{M_i}{N_i}{M_j}{N_j}$ by assumption, we must have by monotonicity that $\nfcl{M_i}{N_i}{A}{N}$. Since $A$ was arbitrary, this implies that $\nfs{M_i}{N_i}{M}{N}$. Similarly, smoothness also holds: fix $(M', N')$ such that $(M_i, N_i) \leanf (M', N')$ for all $i \in I$. We want to show that $\nfs{M}{N}{M'}{N'}$. Fix $A \subseteq U N$ of size less than $\theta$ and fix $i \in I$ such that $A \subseteq U N_i$. By assumption, $\nfs{M_i}{N_i}{M'}{N'}$, so $\nfcl{M_i}{A}{M'}{N'}$, and so by base monotonicity $\nfcl{M}{A}{M'}{N'}$. Since $A$ was arbitrary, the witness property implies that $\nfs{M}{N}{M'}{N'}$, as desired.
    \end{itemize}
  \end{itemize}    
\end{proof}
\begin{remark}
  The proof goes through if we assume only that $\nf$ is invariant and monotonic, and that $\nfm$ is transitive.
\end{remark}

We deduce that having a stable independence notion implies stability and tameness:

\begin{cor}\label{stab-cor}
  Let $\K$ be a $\mu$-AEC with a stable independence relation. Then:

  \begin{enumerate}
  \item (Stability for Galois types) For any $\alpha$, there exists a proper class $S_\alpha$ of cardinals such that for any $\lambda \in S_\alpha$ and $M \in \K_{\lambda}$, $|\gS^{<\alpha} (M)| = \lambda$.
  \item (Tameness) For any $\alpha$, there exists a cardinal $\lambda$ such that for any $M \in \K$, $p, q \in \gS^{<\alpha} (M)$, $p = q$ if and only if $p \rest A = q \rest A$ for all $A \subseteq U M$ with $|A| < \lambda$.
  \end{enumerate}
\end{cor}
\begin{proof}
  Use Theorem \ref{accessible-charact} and imitate the argument in \cite[5.17]{bgkv-apal}, using local character and uniqueness.
\end{proof}

We will see in the next section (Corollary \ref{indep-noop}) that if, in addition, $\K$ has chain bounds, it fails to have a certain order property.

\section{Canonicity, symmetry, and the order property}\label{canon-sec}

In this section, we prove (assuming chain bounds) that stable independence is canonical: there is at most one stable independence relation in any class. In fact, the symmetry property is not necessary, and hence can be deduced from the others. We show further that symmetry implies failure of an order property. Combined with Corollary \ref{stab-cor}, this shows that the class has several features of stable first-order theories.

Note that the proof of symmetry and the order property here are different from those in \cite{bgkv-apal}: since we are working in a more general context than AECs, we do not have Ehrenfeucht-Mostowski models (\cite[4.12]{mu-aec-jpaa}) and hence cannot directly deduce that the order property implies instability as in \cite[5.13]{bgkv-apal}. In fact, this fails in general, see Example \ref{stability-op-example}.

\begin{thm}[The canonicity theorem]\label{canon-thm}
  Let $\K$ be a $\mu$-AEC which has chain bounds (see Definition \ref{directed-def}). Let $\nf$ be an invariant, transitive independence notion on $\nf$ with existence, uniqueness, and right local character.

  Then any other invariant transitive relation on $\K$ satisfying existence, and uniqueness must be $\nf$.
\end{thm}
\begin{proof}[Proof sketch]
  First note that the properties of $\nf$ carry over to $\nfm$, by Fact \ref{nfcl-facts} and Theorem \ref{nfcl-uq}. By Theorem \ref{monster-thm}, we can work inside a monster model $\sea$ which is as homogeneous as we need (if $\K$ does not have joint embedding, we can partition it into subclasses that each have joint embedding by looking at the equivalence classes of the relation ``embedding into a common model''). We want to imitate the argument of \cite[4.14]{bgkv-apal}. It falls into two parts. The first part shows that $\nf$ must have a property there called $(E_+)$ (a strong existence/extension property, which we recall below). The second part shows that having $(E_+)$ implies canonicity. This latter part is implemented in \cite[4.8]{bgkv-apal}, and does not use the fact that $\K$ is an AEC: the argument works in any $\mu$-AEC (and, in fact, in any coherent abstract class).

  It remains to check that $\nf$ has $(E_+)$: for any $M\lea N_0$ and set $A$, there is $N_0 \lea N$ such that for all $N'\cong_{N_0} N$, there is $N_0'\cong_M N_0$ with $\nfcl{M}{A}{N_0'}{\sea}$ and $N_0'\lea N'$ (see \cite[4.4]{bgkv-apal}). The proof of $(E_+)$ uses independent sequences. These are sequences $\seq{A_i : i < \delta}$ (inside a monster model) such that for some $\seq{N_i : i < \delta}$ (called the \emph{witnesses}), $\seq{N_i : i < \delta}$ is an increasing chain, $\bigcup_{j < i} A_j \subseteq U N_i$ and $\nfcl{N_0}{A_i}{N_i}{\sea}$ for all $i < \delta$. We say that $\seq{A_i : i < \delta}$ is \emph{independent over $M$} when there is a witnessing sequence $\seq{N_i : i < \delta}$ with $M = N_0$.

  The first step \cite[4.10]{bgkv-apal} is to show that we can always build independent sequences: given $A$, $M$, and $\delta$, there exists $\seq{A_i : i < \delta}$ that are independent over $M$ and such that the type of each $A_i$ over $M$ is the same as the type of $A$ over $M$ (for some enumerations of $A_i$ and $A$). The argument uses the extension property for $\nfm$ and goes through in the present setup too.

  The second and last ingredient in the proof of $(E_+)$ is to show \cite[4.11]{bgkv-apal} a certain local character property of independent sequences. The proof uses symmetry, so we will instead use the following variation:

  \underline{Claim}: Let $A$ be a set and let $\lambda \ge \mu$ be a regular cardinal such that any type of a sequence of length $|A|$ does not fork over a model of cardinality strictly less than $\lambda$ (this exists by right local character). Then whenever $\seq{M_i : i < \lambda}$ is a $\nf^d$-independent sequence (i.e.\ an independent sequence with respect to $\nf^d$) with $M \lea M_i$ for all $i < \lambda$, then there is $i < \lambda$ with $\nfcl{M}{A}{M_i}{\sea}$.

  \underline{Proof of Claim}: This is the same proof as in \cite[4.11]{bgkv-apal}, but we give it for the convenience of the reader. Let $\seq{N_i : i < \lambda}$ witness the independence and let $N_{\lambda} := \bigcup_{i < \lambda} N_i$. By right local character and base monotonicity, there exists $i < \lambda$ such that $\nfcl{N_i}{A}{N_{\lambda}}{\sea}$. By monotonicity, $\nfcl{N_i}{A}{M_i}{\sea}$. Since the $M_i$'s are $\nf^d$-independent, we also have that $\nfcl{M}{N_i}{M_i}{\sea}$. Using left transitivity, we get that $\nfcl{M}{A}{M_i}{\sea}$, as desired. $\dagger_{\text{Claim}}$

  Now that the claim is proven, the argument of \cite[4.13]{bgkv-apal} goes through to show that $\nf$ has $(E_+)$, completing the proof.
\end{proof}
\begin{remark}
  It is enough to assume that $\K$ has monster models instead of having chain bounds, see Theorem \ref{monster-thm}.
\end{remark}

We deduce the symmetry property:

\begin{cor}\label{sym-cor}
  If $\K$ is a $\mu$-AEC which has chain bounds and $\nf$ is an invariant, transitive independence notion with existence, uniqueness and right local character, then $\nf$ is symmetric. Thus if $\nf$ has, in addition, the right (or left) witness property, it is a stable independence relation.
\end{cor}
\begin{proof}
  $\nf^d$ is invariant, transitive, and has existence and uniqueness, so by Theorem \ref{canon-thm} $\nf = \nf^d$. The last sentence follows from Theorem \ref{accessible-charact}.
\end{proof}

\begin{cor}\label{canon-cor}
  Let $\ck$ be a category which has chain bounds and whose morphisms are monomorphisms. Then there is at most one stable independence notion on $\ck$.
\end{cor}
\begin{proof}
  By Lemma \ref{acc}, the equivalence between $\mu$-AECs and accessible categories with all morphisms monomorphisms, Theorem \ref{accessible-charact}, and Theorem \ref{canon-thm}.
\end{proof}

\begin{question}\label{canon-q}
  Can one prove an even more general canonicity result? What if the morphisms of the category are not all monomorphisms? What if the category does not have chain bounds?
\end{question}

The following remark shows that a weakening of chain bounds is enough for the proof of the canonicity theorem to go through.

\begin{remark}\label{canon-locally-finite}
  Let $\kappa$ be an infinite cardinal and consider the class $\ck$ of $\kappa$-local graphs studied in Example \ref{locally-finite-example}. Then $\ck$ has a stable independence notion $\nf$ but does \emph{not} have chain bounds. However the proof of Theorem \ref{canon-thm} still goes through: chain bounds are only used to build independent sequences, and, at least in this case, we can achieve the same effect by other methods. Let $A$ be a $\kappa$-local graph and let $M$ be a fixed $\kappa$-local graph. For simplicity, assume that $M$ is a full subgraph of $A$. We build $\seq{A_i :i < \delta}$ a sequence of $\kappa$-local graphs with $A_i \cong_M A$ for all $i < \delta$, as well as $\seq{N_i : i < \delta}$ an increasing continuous sequence of $\kappa$-local graphs such that $N_0 = M$, $N_{i + 1} = N_i \cup A_i$ and $\nfs{M}{A_i}{N_i}{N_{i + 1}}$ for $i < \delta$. Using the existence property, it is easy to implement the successor step. For the limit step, it is enough to check that for $i$ limit, $N_i = \bigcup_{j < i} N_j$ is again $\kappa$-local. So let $v$ be a vertex of $N_i$. Pick $j < i$ such that $v \in N_j$. By construction, there are no edges from $v$ to $A_{j'}$ for $j' \ge j$, and moreover $N_i = \bigcup_{j' < i} A_{j'}$, so any edge of $v$ in $N_i$ must be contained in $N_j$, which is $\kappa$-local, as desired.
\end{remark}

We will use the following definition of the order property, introduced by Shelah for AECs \cite[4.3]{sh394}. In the first-order case, it is equivalent to the usual definition.

\begin{defin}\label{op-def}
  A $\mu$-AEC $\K$ has the \emph{$\alpha$-order property of length $\theta$} if there exists $M \in \K$ and a sequence $\seq{\ba_i : i < \theta}$ such that:
  
  \begin{enumerate}
  \item $\ba_i  \in \fct{\alpha}{U M}$ for all $i < \theta$.
  \item For all $i_0 < j_0 < \theta$ and $i_1 < j_1 < \theta$, $\gtp (\ba_{i_0} \ba_{j_0} / \emptyset; M) \neq \gtp (\ba_{j_1} \ba_{i_1} / \emptyset; M)$.
  \end{enumerate}

  $\K$ has the \emph{order property} if there exists $\alpha$ such that for all $\theta$, $\K$ has the $\alpha$-order property of length $\theta$.
\end{defin}

\begin{thm}\label{op-thm}
  Let $\K$ be a $\mu$-AEC and $\nf$ be an invariant, transitive independence notion with existence, uniqueness, and right local character. If $\K$ has chain bounds, or more generally if $\nf$ has symmetry, then $\K$ does \emph{not} have the order property.
\end{thm}
\begin{proof} 
  Recall from Corollary \ref{sym-cor} that chain bounds implies symmetry in the context of the theorem. Assume for a contradiction that $\K$ has the order property. Pick $\alpha$ and $\theta$ such that $\K$ has the $\alpha$-order property of length $\theta^+$, with $\theta = \theta^{<\mu} \ge \LS (\K)$ ``sufficiently big'' (the proof will tell us how big) such that $|\gS^\alpha (M)| \le \theta$ for all $M \in \K_{\le \theta}$ (exists by the proof of Corollary \ref{stab-cor}). Pick $\seq{\ba_i : i < \theta^+}$ and $M$ witnessing the $\alpha$-order property of length $\theta^+$. Let $\seq{M_i : i < \theta^+}$ be an increasing sequence of submodels of $M$ such that $\{\ba_j : j < i\} \subseteq U M_{i}$ for all $i < \theta^+$ and $|U M_i| \le \theta$.   For each $i < \theta^+$ of sufficiently high cofinality, there exists $j < i$ such that $\gtp (\ba_i / M_i)$ does not fork over $M_j$. By Fodor's lemma, we can therefore assume without loss of generality that $\gtp (\ba_i / M_i; M)$ does not fork over $M_0$ for all $i < \theta^+$. Since $|\gS^\alpha (M_0)| \le \theta$, we can further assume without loss of generality that $\gtp (\ba_i / M_0; M) = \gtp (\ba_{i'} / M_0; M)$ for all $i < i' < \theta^+$. By uniqueness, this implies that $\gtp (\ba_i / M_i; M) = \gtp (\ba_{i'} / M_i; M)$.

  We are now in the following setup: $\ba_0 \ba_1$ is a two-element independent sequence over $M_0$ (in this case, this means that $\nfcl{M_0}{\ba_0}{\ba_1}{M}$), and so by symmetry also $\nfcl{M_0}{\ba_1}{\ba_0}{M}$, so $\ba_1 \ba_0$ is a two-element independent sequence over $M_0$. Now the proof of \cite[4.8]{tame-frames-revisited-jsl} (specifically the successor step) tells us that if we have two two-element independent sequences $\ba \bb$ and $\ba' \bb'$ over a model $M_0$, so that the types of their individual elements are equal over $M_0$ (i.e.\ $\ba$ and $\ba'$ agree over $M_0$ and $\bb$ and $\bb'$ agree over $M_0$), then in fact $\ba \bb$ and $\ba' \bb'$ have the same type over $M_0$. Applying this here, we get that $\gtp (\ba_0 \ba_1 / M_0; M) = \gtp (\ba_1 \ba_0/ M_0; M)$. This contradicts that the sequence $\seq{\ba_i : i < \theta^+}$ witnessed the order property. 
\end{proof}

\begin{cor}\label{indep-noop}
  If $\K$ is a $\mu$-AEC with a stable independence relation, then $\K$ does not have the order property. 
\end{cor}
\begin{proof}
  This is a consequence of Theorem \ref{op-thm}, using the equivalence given by Theorem \ref{accessible-charact}.
\end{proof}

\begin{remark}\label{banachord-rmk} We may now give an alternate proof of the fact that $\Ban$, the category of Banach spaces and linear contractions, does not have effective unions (see Example~\ref{effective-union-examples}(4)). If $\Ban$ had effective unions, the corresponding $\mu$-AEC $\K$ of Banach spaces with isometries would have a stable independence notion by Theorem \ref{indep}. By Corollary \ref{indep-noop}, this means that $\K$ does not have the order property (see Definition \ref{op-def}). Take however the Banach space $c_0$ of complex-valued sequences $\seq{a_n : n < \omega}$ going to zero with the supremum norm. Let $e_n$ be the sequence that is one at position $n$ and 0 elsewhere. Let $f_n := \sum_{i \le n} e_i$. Then $\|e_m + f_n\| = 2$ if and only if $m \le n$ (see \cite{krivine-maurey}). Thus $c_0$ satisfies an instance of the order property of length $\omega$. By the compactness theorem for continuous first-order logic (see e.g.\ \cite{fourguys-metric}), this implies that $\K$ must have the order property.  Contradiction.\end{remark}

Shelah \cite[\S1]{sh1019} examines several definitions of stability for $\Ll_{\kappa, \kappa}$, $\kappa$ a strongly compact cardinal, and shows that, while there are natural implications, there are also several non-equivalences. The following is a simple example:

\begin{example}\label{stability-op-example}
  Let $\K$ be the $\aleph_1$-AEC of well-orderings, ordered by being a suborder. Then $\K$ has the order property but for every $M \in \K$, $|\gS^{<\omega} (M)| = |U M| + \aleph_0$. 
\end{example}

On the other hand it is known that failure of the order property implies stability in terms of counting Galois types. Roughly, this is because the proof of the corresponding first-order fact can be carried out inside a fixed model, see \cite[\S V.A]{shelahaecbook2}, \cite[\S1]{sh1019}, or the proof of \cite[4.11]{sv-infinitary-stability-afml}.

\section{Stable independence on saturated models}\label{noop-sec}

Putting together \cite{bg-apal, bgkv-apal, sv-infinitary-stability-afml}, one obtains the following converse to Corollary \ref{indep-noop}: Assuming large cardinals, any $\mu$-AEC which has chain bounds and does not have the order property admits a stable independence relation on a subclass of model-homogeneous models. This was essentially observed for AECs with amalgamation by Boney and Grossberg \cite[8.2]{bg-apal}, although categoricity is used there to prove local character. We work from the result of \cite[\S5]{sv-infinitary-stability-afml}, which focuses on the stable case and avoids any use of categoricity. Note that it is in general necessary to pass to a subclass of model-homogeneous models, see Example \ref{indep-examples}(\ref{field-example}).

\begin{thm}[The existence theorem]\label{indep-build}
  Assume Vop\v enka's principle. Let $\K$ be a $\mu$-AEC which has chain bounds. Let $\kappa > \LS (\K)$ be strongly compact. If $\K$ does not have the order property, then the $\kappa$-AEC $\Kmh$ of locally $\kappa$-model-homogeneous models of $\K$ (see Definition \ref{mh-def}) has a stable independence relation.
\end{thm}
\begin{proof}[Proof sketch]
  By Theorem \ref{mh-thm}, $\Kmh$ has amalgamation and is indeed a $\kappa$-AEC. First observe that for any $\alpha < \kappa$, $\K$ does not have the $\alpha$-order property of length $\kappa$. If it did, then we would be able to take repeated ultraproducts of the universe to make the sequence longer and get that $\K$ has the $\alpha$-order property of any length. Moreover, the strongly compact also gives that $\K$ is fully $(<\kappa)$-tame and short (see \cite[5.5]{mu-aec-jpaa}).

  Now define $\nfs{M_0}{M_1}{M_2}{M_3}$ to hold if and only if $M_0 \lea M_\ell \lea M_3$, $\ell = 1,2$, all are in $\Kmh$, and for any $\ba \in \fct{<\kappa}{M_1}$, $\gtp (\ba / M_2; M_3)$ is a $\kappa$-coheir over $M_0$. That is, for any $A \subseteq U M_2$ of cardinality less than $\kappa$, $\gtp (\ba / A; M_3)$ is realized inside $M_0$.

  By the proofs of \cite{bg-apal} or \cite[\S5]{sv-infinitary-stability-afml}, we get all the required properties. More precisely, it is straightforward to check that $\nf$ is invariant and monotonic. Transitivity, symmetry, and uniqueness are proven as in \cite[\S4]{bg-apal} or \cite[5.15]{sv-infinitary-stability-afml}. The $(<\kappa)$-witness property then follows from the definition. Local character is \cite[5.15(2c)]{sv-infinitary-stability-afml} and existence is \cite[8.2]{bg-apal}. Using Theorem \ref{accessible-charact}, we get that $\nf$ is stable, as desired.
\end{proof}
\begin{remark}
  The assumption of Vop\v enka's principle can be removed if $\K$ has amalgamation.
\end{remark}

We summarize:

\begin{cor}\label{indep-build-cor}
  Assume Vop\v enka's principle. Let $\K$ be a $\mu$-AEC which has chain bounds. The following are equivalent:

  \begin{enumerate}
  \item $\K$ does not have the order property.
  \item For some sub-$\lambda$-AEC $\K^\ast \subseteq \K$ which is cofinal in $\K$ and has the same ordering as $\K$, $\K^\ast$ has a stable independence notion.
  \end{enumerate}
\end{cor}
\begin{proof}
  Combine Theorems \ref{op-thm} (noting that if $\K$ has the order property, then any of its cofinal subclasses must have it) and \ref{indep-build}. 
\end{proof}

Note that in general, one can always restrict a stable independence relation to a subclass of model-homogeneous models. In fact:

\begin{lem}
  Let $\K$ be a $\mu$-AEC with a stable independence relation $\nf$. Let $\K^\ast$ be a subclass of $\K$ (ordered by the appropriate restriction of $\lea$) satisfying the following conditions:

  \begin{enumerate}
  \item $\K^\ast$ is a $\lambda$-AEC, for some $\lambda$.
  \item $\K^\ast$ is cofinal in $\K$ (that is any $M \in \K$ extends to an $N \in \K^\ast$).
  \end{enumerate}

  Then the restriction $\nf^\ast$ of $\nf$ to $\K^\ast$ is a stable independence relation on $\K^\ast$.
\end{lem}
\begin{proof}[Proof sketch]
  We use Theorem \ref{accessible-charact}. It is straightforward to check that $\nf^\ast$ is an invariant, monotonic, symmetric, and transitive independence relation with the witness property. Existence and uniqueness follow from the corresponding properties of $\nf$ and the fact that $\K^\ast$ is cofinal in $\K$. To see local character, we use local character of $\nf$ together with the fact that $\K^\ast$ is a $\lambda$-AEC, hence satisfies the Löwenheim-Skolem-Tarski axiom.
\end{proof}

Restricting to a subclass of sufficiently homogeneous models, we can also get that any stable square is a pullback square (i.e.\ is disjoint over the base). The argument is similar to \cite[12.13]{indep-aec-apal} but we give a self-contained version here.

\begin{lem}\label{algebraic-nf-lem}
  Let $\nf$ be a stable independence relation on a $\mu$-AEC $\K$. Let $\lambda > \LS (\K)$ and let $M \lea N$ both be in $\K$ with $M$ $\lambda$-model-homogeneous. Let $p \in \gS (N)$ and assume that $p$ does not fork over $M_0$, with $M_0 \lea M$ and $|U M_0|^{<\mu} < \lambda$. Then $p$ is algebraic (i.e.\ realized inside its domain) if and only if $p \rest M$ is algebraic. 
\end{lem}
\begin{proof}
  If $p \rest M$ is algebraic, then $p$ is clearly algebraic. Conversely, assume that $p$ is algebraic. Pick $a \in N$ realizing $p$ and let $N_0 \in \K_{<\lambda}$ be such that $M_0 \lea N_0 \lea N$ and $N_0$ contains $a$. Since $M$ is $\lambda$-model-homogeneous, there exists $f: N_0 \rightarrow M$ fixing $M_0$. Now by monotonicity $p \rest N_0$ does not fork over $M_0$, so $f (p \rest N_0)$ does not fork over $M_0$. Since $f$ fixes $M_0$, $p \rest M_0 = f (p \rest N_0) \rest M_0$, so by uniqueness $f (p \rest N_0) = p \rest f[N_0]$. Since $p \rest N_0$ is algebraic, this implies that $p \rest f[N_0]$ also is. Since $f[N_0] \lea M$, $p \rest M$ is algebraic, as desired.
\end{proof}

\begin{thm}\label{indep-pullback}
  Let $\nf$ be a stable independence relation on a $\mu$-AEC $\K$. There exists a regular cardinal $\lambda > \LS (\K)$ such that if $\nfs{M_0}{M_1}{M_2}{M_3}$ and $M_0$ is $\lambda$-model-homogeneous, then $M_1 \cap M_2 = M_0$. 
\end{thm}
\begin{proof}
  By local character, we can pick $\lambda > \LS (\K)$ regular so that $\lambda_0^{<\mu} < \lambda$ for all $\lambda_0 < \lambda$ and moreover any type of length one does not fork over a model of cardinality strictly less than $\lambda$. Now assume that $\nfs{M_0}{M_1}{M_2}{M_3}$ and $M_0$ is $\lambda$-model-homogeneous. Let $a \in M_1 \cap M_2$. We show that $a \in M_0$. Let $p := \gtp (a / M_2; M_3)$. Note that $p$ is algebraic. By how $\lambda$ was chosen, there exists $M_0' \lea M_0$ such that $M_0' \in \K_{<\lambda}$ and $p \rest M_0$ does not fork over $M_0'$. By transitivity, $p$ does not fork over $M_0'$. By Lemma \ref{algebraic-nf-lem} (where $M_0, M, N$ there stand for $M_0', M_0, M_2$ here), $p \rest M_0$ is algebraic, so $a \in M_0$, as desired.
\end{proof}

\bibliographystyle{amsalpha}
\bibliography{indep-categ}

\end{document}